\documentclass[twoside,english,british,draft,11pt]{amsart}
\usepackage[T1]{fontenc}
\usepackage[latin9]{inputenc}
\usepackage{verbatim}
\usepackage{amsthm}
\usepackage{amstext}
\usepackage{amssymb}
 \makeatletter
 \numberwithin{equation}{section}
 \numberwithin{figure}{section}
 \theoremstyle{plain}
 \newtheorem{thm}{\protect\theoremname}[section]
   \theoremstyle{plain}
   \newtheorem{conjecture}[thm]{\protect\conjecturename}
   \theoremstyle{plain}
   \newtheorem{lem}[thm]{\protect\lemmaname}
   \theoremstyle{plain}
   \newtheorem{cor}[thm]{\protect\corollaryname}
   \theoremstyle{remark}
   \newtheorem{exa}[thm]{Example}%{\protect\examplename}
   \theoremstyle{remark}
   \newtheorem{rem}[thm]{\protect\remarkname}
   \theoremstyle{plain}
   \newtheorem{prop}[thm]{\protect\propositionname}
   \theoremstyle{definition}
   \newtheorem{defn}[thm]{\protect\definitionname}
   \theoremstyle{remark}
   \newtheorem*{acknowledgement*}{\protect\acknowledgementname}

 \newcommand{\mcDn}{\mathcal{D}_n}
 \newcommand{\cb}{\mathcal{C}_n}  % ChessBoard
 \newcommand{\cbe}{\mathcal{C}_{n,0}}  % ChessBoard Even
 \newcommand{\cbo}{\mathcal{C}_{n,1}}  % ChessBoard Odd
 \newcommand{\firstcb}{\mathcal{M}_n}  % First ChessBoard factorisation
 \newcommand{\firste}{\mathcal{E}_n}  % First Even factorisation

 \newcommand{\imin}{e}

 \usepackage{array}\usepackage{float}\usepackage{amsthm}\usepackage{amsbsy}%, showkeys}

 \usepackage[normalem]{ulem}

 \theoremstyle{plain}
 \newcommand{\mfp}{\mathfrak{p}}

 \newcommand{\ul}{\underline}

 \newcommand{\Fq}{\ensuremath{\mathbb{F}_q}}
 \newcommand{\N}{\ensuremath{\mathbb{N}}}

 \newcommand{\Z}{\ensuremath{\mathbb{Z}}}

 \newcommand{\bfG}{{\bf G}}

 \newcommand{\calO}{\ensuremath{\mathcal{O}}}

 \newcommand{\rk}{\ensuremath{{\rm rk}}}

 \newcommand{\Gri}{\ensuremath{\mathcal{O}}}

 \DeclareMathOperator{\Id}{Id}

 \DeclareMathOperator{\Mat}{Mat}

 \DeclareMathOperator{\coxinv}{inv}
 \DeclareMathOperator{\coxneg}{neg}
 \DeclareMathOperator{\coxnsp}{nsp}
 
 \renewcommand{\epsilon}{\varepsilon}
 \renewcommand{\phi}{\varphi}

 \makeatother

 \usepackage{babel}
   \addto\captionsbritish{\renewcommand{\acknowledgementname}{Acknowledgement}}
   \addto\captionsbritish{\renewcommand{\conjecturename}{Conjecture}}
   \addto\captionsbritish{\renewcommand{\corollaryname}{Corollary}}
   \addto\captionsbritish{\renewcommand{\definitionname}{Definition}}
   \addto\captionsbritish{\renewcommand{\lemmaname}{Lemma}}
   \addto\captionsbritish{\renewcommand{\propositionname}{Proposition}}
   \addto\captionsbritish{\renewcommand{\remarkname}{Remark}}
   \addto\captionsbritish{\renewcommand{\theoremname}{Theorem}}
   \addto\captionsenglish{\renewcommand{\acknowledgementname}{Acknowledgement}}
   \addto\captionsenglish{\renewcommand{\conjecturename}{Conjecture}}
   \addto\captionsenglish{\renewcommand{\corollaryname}{Corollary}}
   \addto\captionsenglish{\renewcommand{\definitionname}{Definition}}
   \addto\captionsenglish{\renewcommand{\lemmaname}{Lemma}}
   \addto\captionsenglish{\renewcommand{\propositionname}{Proposition}}
   \addto\captionsenglish{\renewcommand{\remarkname}{Remark}}
   \addto\captionsenglish{\renewcommand{\theoremname}{Theorem}}
   \providecommand{\acknowledgementname}{Acknowledgement}
   \providecommand{\conjecturename}{Conjecture}
   \providecommand{\corollaryname}{Corollary}
   \providecommand{\definitionname}{Definition}
   \providecommand{\lemmaname}{Lemma}
   \providecommand{\propositionname}{Proposition}
   \providecommand{\remarkname}{Remark}
 \providecommand{\theoremname}{Theorem}

\begin{document}

 \title{A new statistic on the hyperoctahedral groups}

 \author{Alexander Stasinski and Christopher Voll}
 \begin{abstract}
   We introduce a new statistic on the hyperoctahedral groups (Coxeter
   groups of type~$B$), and give a conjectural formula for its signed
   distributions over arbitrary descent classes.  The statistic is
   analogous to the classical Coxeter length function, and features a
   parity condition. For descent classes which are singletons the
   conjectured formula gives the Poincar\'e polynomials of the
   varieties of symmetric matrices of fixed rank.

   For several descent classes we prove the conjectural formula. For
   this we construct suitable ``supporting sets'' for the relevant
   generating functions. We prove cancellations on the complements of
   these supporting sets using suitably defined sign reversing
   involutions.
 \end{abstract}
 \selectlanguage{english}%

 \address{Alexander Stasinski, School of Mathematics, University of
   Southampton, University Road, Southampton SO17 1BJ, United Kingdom}
 \curraddr{Department of Mathematical Sciences, Durham University,
   South Road, Durham, DH1 3LE UK} \email{alexander.stasinski@durham.ac.uk}

 \address{Christopher Voll, School of Mathematics, University of
   Southampton, University Road, Sou\-thampton SO17 1BJ, United Kingdom}
 \curraddr{Fakult\"at f\"ur Mathematik, Universit\"at Bielefeld,
   Postfach 100131, 33501 Bielefeld, Germany}

 \email{C.Voll.98@cantab.net}

 %\thanks{This file is called \boxed{\text{\jobname}} \hfill{}\textbf{Date
 %of draft version: \today}}

 \selectlanguage{british}%


 \selectlanguage{english}%


 \maketitle
 \selectlanguage{british}%
 %\tableofcontents{}

 \thispagestyle{empty}
 \section{Introduction\label{sec:Introduction}}
 There is an extensive literature concerned with identities for
 generating functions for~$S_{n}$, the symmetric group of degree~$n$.
 These are typically (multi-variable) polynomials obtained by summing
 the values of $\N_0$-valued functions, or \emph{statistics}, on the
 Coxeter group~$S_n$. Sometimes the sums are twisted with the
 non-trivial linear character of $S_n$.  Occasionally, one can prove
 more refined versions where the sums are restricted to descent
 classes. Recently, there has been an interest in finding
 generalisations, or suitable analogues, of such results for the
 hyperoctahedral groups; see for example
 \cite{Reiner-Signed-perm,AdinBrentiRoichman/01,AdinBrentiRoichman/06,FoataHan/09}.
 The \emph{hyperoctahedral group} $B_{n}$ is the group of permutations
 $w$ of the set $[\pm n]_{0}:=\{-n,\dots,n\}$ such that $w(-j)=-w(j)$
 for all~$j\in[\pm n]_{0}$.

 \selectlanguage{english}%
In the present paper we study generating
functions involving a new statistic $L$ on $B_{n}$.  For $w\in B_{n}$
we define
\begin{equation}\label{def:L}
 L(w)=\frac{1}{2}\#\{(i,j)\in[\pm n]_{0}^{2}\mid i<j,\ w(i)>w(j),\ i\not\equiv j\bmod{(2)}\}\in \{0\} \cup \N.
\end{equation}

To state our results, we introduce some further notation. Let $\N$
denote the set of positive integers, and~$\N_{0}=\{0\}\cup\N$.  For
$n\in\N$, let $[n]=\{1,2,\dots,n\}$ and $[n]_{0}=\{0\}\cup[n]$.  We
write $(\ul{n})_X$ or $(\ul{n})$ for the polynomial
$1-X^{n}\in\mathbb{Z}[X]$, where $X$ is an indeterminate. We set
$(\underline{0})=1$ and write $(\ul{n})_X!$ or $(\underline{n})!$ for
$(\underline{1})(\underline{2})\cdots(\underline{n})$. For a real
number $x$, we write $\lfloor x\rfloor$ for the largest integer less
than or equal to~$x$. Let
$I=\{i_{1},\dots,i_{l}\}_{<}\subseteq[n-1]_{0}$, that is
$i_1<\dots<i_l$. We put $i_{1}=\min(I\cup\{n\})$ and~$i_{l+1}=n$,
respectively. Let $S=\{s_{0},\dots,s_{n-1}\}$ be the set of Coxeter
generators for $B_{n}$ described in
\cite[Section~8.1]{BjoernerBrenti/05} (see also
Section~\ref{sec:preliminaries}) and let $l:B_{n}\rightarrow\N_{0}$
denote the (Coxeter) length function on $B_{n}$ with respect
to~$S$. We define the \emph{quotient} (or \emph{descent
  class}) $$B_{n}^{I}=\{w\in B_{n}\mid D(w)\subseteq
I^{\mathrm{c}}\},$$ where $D(w):=\{i\in[n-1]_0 \mid l(ws_{i})<l(w)\}$
denotes the (right) \emph{descent set} of~$w$ and $I^{\mathrm{c}}$
denotes the complement~$[n-1]_{0}\setminus I$; cf.~\cite[Sections 2.4
  and 8.1]{BjoernerBrenti/05}. Thus $w\in B_n^{I^{\mathrm{c}}}$ if and
only if $D(w)\subseteq I$.  For $n\in\N$ and
$I=\{i_1,\dots,i_l\}_<\subseteq[n-1]_0$ we define the polynomials
\begin{equation*}
  f_{n,I}(X)=\frac{(\underline{n})!}{(\underline{i_{1}})!}\prod_{r=1}^{l}\prod_{
    \sigma=1}^{ \lfloor (i_{r+1}-i_{r})/2
    \rfloor}(\underline{2\sigma})^{-1}\in\Z[X].
\label{polysf}
\end{equation*}
In \cite{StasinskiVoll/11} we stated the following conjecture:
\selectlanguage{british}%
\begin{conjecture}\cite[Conjecture~1.6]{StasinskiVoll/11}
  \label{conjecture}For $n\in\N$ and
  $I=\{i_1,\dots,i_l\}_<\subseteq[n-1]_0$,
\begin{equation}
\sum_{w\in
  B_{n}^{I^{\mathrm{c}}}}(-1)^{l(w)}X^{L(w)}=f_{n,I}(X).\label{eq:def}
\end{equation}

\end{conjecture}
For instance, if $I=[n-1]_{0}$ then $B_{n}^{I^{\mathrm{c}}}=B_{n}$, and
formula \eqref{eq:def} reads
\[
\sum_{w\in B_{n}}(-1)^{l(w)}X^{L(w)}=(\underline{n})!.
\]

\selectlanguage{english} Our main result is the following.
\begin{thm}
\label{thm:Main}
Conjecture~\ref{conjecture} holds in the following cases:
\begin{enumerate}
\item \label{1}$n\in\N$ and \foreignlanguage{english}{$I=\{0\}$,}
\item \label{2}$n\in\N$ and \foreignlanguage{english}{$I=[n-1]_{0}$,}
\selectlanguage{english}%
\item \label{3}$n\in2\N$ and $I\subseteq[n-1]_{0}\cap2\mathbb{Z}$.
\end{enumerate}
\end{thm}
The three parts of Theorem~\ref{thm:Main} are proved in
Sections~\ref{sec:The-case1}-\ref{sec:The-case3}, namely
Propositions~\ref{pro:Ascending}, \ref{pro:Perm-case}, and
\ref{pro:Even}.  Our methods are based on defining supporting sets for
the sums in question, and sign reversing involutions on their
complements which preserve their intersections with the descent
classes $B_{n}^{I^{\mathrm{c}}}$ and leave $L$ invariant. The sets
$B_{n}^{I^{\mathrm{c}}}$ in \eqref{eq:def} may thus be replaced by
their intersections with the supporting sets; the contributions of the
other elements to the sums cancel out. On the supporting sets the
statistic $L$ behaves better than on the whole of
$B_{n}^{I^{\mathrm{c}}}$: in Section~\ref{sec:case3} we establish, for
instance, two additivity results for $L$ with respect to certain
parabolic factorisations.

For one-element sets $I=\{i\}$, where $i\in[n-1]_0$, the polynomials
$f_{n,\{i\}}$ yield the Poincar\'e polynomials of the varieties of
symmetric $n\times n$ matrices over $\Fq$ of rank $n-i$. Indeed, it is
well known that, for all prime powers~$q$,
$$\#\{x\in\Mat_n(\Fq) \mid x = x^{\mathrm{t}}, \, \rk(x) = n-i\} =
q^{\binom{n+1}{2}-\binom{i+1}{2}}f_{n,\{i\}}\,(q^{-1});$$ see, for
instance, \cite[Lemma~10.3.1]{Igusa/00} and compare~\cite[Lemma~3.1
(3.4)]{StasinskiVoll/11}. It is interesting whether -- at least in
these cases -- Conjecture~\ref{conjecture} reflects cohomological
properties of the varieties of symmetric matrices of fixed rank.

The restriction of $L$ to $S_{n}$ agrees with the function $L$ defined
in~\cite[Definition~5.1]{KlopschVoll/09}. In fact, Conjecture
\ref{conjecture} may be seen as a type-$B$-analogue of
\cite[Conjecture~C]{KlopschVoll/09}. The polynomials in this
conjecture encode the numbers of non-degenerate flags in finite vector
spaces equipped with a non-degenerate quadratic form.

\selectlanguage{british} The results in the current paper are mainly
motivated by our work \cite{StasinskiVoll/11} on representation zeta
functions of nilpotent groups. In the remainder of the introduction we
describe this connection briefly. Let $G$ be a
\foreignlanguage{english}{finitely generated, torsion-free nilpotent
  group. The \emph{representation zeta function }of $G$ is the
  Dirichlet generating series
\[
\zeta_{G}(s):=\sum_{n=1}^{\infty}\widetilde{r}_{n}(G)n^{-s},
\]
where $s$ is a complex variable, and $\widetilde{r}_{n}(G)$ denotes
the number of $n$-dimensional irreducible complex representations
of~$G$, up to twisting by $1$-dimensional representations. }In
\cite[Theorem~C]{StasinskiVoll/11}, the representation zeta functions
are explicitly computed for three infinite families of groups of
nilpotency class~$2$, namely $F_{2n+\eta}(\calO)$, $G_{2n}(\calO)$,
$H_{2n}(\calO)$, where $n\in\N$, $\eta\in\{0,1\}$, and $\calO$ is the
ring of integers in an arbitrary number field. When $2n+\eta=2n=2$
these groups all coincide with the Heisenberg group of $3\times3$
upper unitriangular matrices over $\calO$. Let $\bfG$ denote any of
the group schemes $F_{2n+\eta}$, $G_{2n}$ or~$H_{2n}$. It can be shown
that $\zeta_{\bfG(\calO)}(s)$ has an Euler product 
\[
\zeta_{\bfG(\Gri)}(s)=\prod_{\mfp}\zeta_{\bfG(\Gri_{\mfp})}(s),
\]
where $\mfp$ runs through the non-zero prime ideals of $\calO$ and
$\calO_{\mfp}$ denotes the completion of $\calO$ at $\mfp$, and that
each local factor $\zeta_{\bfG(\Gri_{\mfp})}(s)$ is a rational
function in $q^{-s}$, where $q=|\calO/\mfp|$ is the residue field
cardinality at $\mfp$. In fact, these properties hold much more
generally; see~\cite[Proposition~2.2 and
Corollary~2.19]{StasinskiVoll/11}.  In \cite{StasinskiVoll/11} we
showed that the local zeta functions $\zeta_{\bfG(\Gri_{\mfp})}(s)$
are related to $q$-series and statistics on hyperoctahedral
groups. More precisely, \cite[Theorem~C]{StasinskiVoll/11} states that
there exist a family of polynomials
$(f_{\mathbf{G},I}(X))_{I\subseteq[n-1]_{0}}$ in $\Z[X]$ and integers
$(a(\mathbf{G},i))_{i\in[n-1]_{0}}$ such that, for all~$\mfp$,
\begin{equation}
  \zeta_{\mathbf{G}(\calO_{\mfp})}(s)=\sum_{I\subseteq[n-1]_{0}}f_{\mathbf{G},I}(q^{-1})\prod_{i\in I}\frac{q^{a(\mathbf{G},i)-(n-i)s}}{1-q^{a(\mathbf{G},i)-(n-i)s}}.\label{eq:zeta-additive}
\end{equation}
The polynomials $f_{\mathbf{G},I}(X)$ turn out to have a combinatorial
interpretation: in
\cite[Proposition~4.6]{StasinskiVoll/11} we showed that, for
$I\subseteq[n-1]_{0}$, %one has
\begin{align}
f_{F_{2n+\eta},I}(X) & =\sum_{w\in B_{n}^{I^{\mathrm{c}}}}(-1)^{\coxneg(w)}X^{2l(w)+(2\eta-1)\coxneg(w)},\label{equ:F}\\
f_{G_{2n},I}(X) & =\sum_{w\in B_{n}^{I^{\mathrm{c}}}}(-1)^{\coxneg(w)}X^{l(w)}.\label{equ:G}
\end{align}
Here $\coxneg(w):=\#\{i\in[n]\mid w(i)<0\}$ for $w\in B_n$.  Key to
the equations \eqref{equ:F} and \eqref{equ:G} are formulae for the
joint distributions of the statistics $\coxneg$ and $l$ on descent
sets of $B_n$ which were given by V.~Reiner;
cf.~\cite[Lemma~4.5]{StasinskiVoll/11}. For the group
schemes~$H_{2n}$, we know that
\[
f_{H_{2n},I}(X)=f_{n,I}(X)
\]
(cf.~\cite[Theorem~C]{StasinskiVoll/11}) and
Conjecture~\ref{conjecture} is a conjectural analogue of \eqref{equ:F}
and \eqref{equ:G}.  Combinatorial formulae of the
form~\eqref{eq:def} often have interesting consequences for zeta
functions of the form~\eqref{eq:zeta-additive}. In particular, such
formulae may facilitate proofs that the corresponding zeta function
satisfy functional equations; see \cite[Theorem~B]{KlopschVoll/09}.

\section{Signed permutations, chessboard elements and supporting sets}\label{sec:preliminaries}

Throughout, we keep the notation introduced in
Section~\ref{sec:Introduction}.  Let $W$ be a Coxeter group with
Coxeter generating set $S$.  For $I\subseteq S$, we denote by
$W_{I}=\langle s_{i}\mid i\in I\rangle$ the corresponding standard
parabolic subgroup of $W$. We also introduce the quotient
$W^{I}:=\{w\in W\mid D(w)\subseteq I^{\mathrm{c}}\}$. It is well known
that every element $w\in W$ has a unique factorisation (or ``parabolic
decomposition'')
\begin{equation}\label{equ:factorisation}
w=w^{I}w_{I},\quad\text{ where \ensuremath{w^{I}\in W^{I}}and \ensuremath{w_{I}\in W_{I}}}.
\end{equation}
The elements of $W^{I}$ are the unique representatives of the cosets
in $W/W_{I}$ of shortest length. The Coxeter length function $l$ on
$W$ is additive with respect to this factorisation, that is
\begin{equation}\label{equ:add.length}
  l(w)=l(w^I)+l(w_I);
\end{equation}
see \cite[Section 1.10]{Humphreys/90}.

%The lengths\footnote{CV: What's
%  the point of this sentence? Suggest to cut out.} of the elements
%$w^{I}$, for various $I$, play a key role in the ``parabolic length
%functions'' studied in~\cite{KlopschVoll/09}.

Let now, specifically, $W$ be the hyperoctahedral group~$B_{n}$.  This
Coxeter group has a concrete combinatorial description, which we now
recall; cf.~\cite[Section~8.1]{BjoernerBrenti/05}. The group $B_n$ has
a faithful representation which identifies it with the group of
``signed permutation matrices'', that is, monomial $n\times n$
matrices with non-zero entries in $\{-1,1\}$, acting on standard basis
column vectors and their negatives. For $w\in B_n$ we use the ``window
notation'' $w=[a_{1},\dots,a_{n}]$ to mean that, for $i\in[n]$,
$w(i)=a_{i}\in[\pm n]_0$. In this notation, define
\begin{align*}
s_{i}&=[1,\dots,i-1,i+1,i,i+2,\dots,n] \text{ for $i\in[n-1]$ and}\\
s_{0}&=[-1,2,\dots,n].
\end{align*}
The set $S := \{s_0,s_1,\dots,s_{n-1}\}$ is a set of Coxeter generators
for $B_n$. The Coxeter length function with respect to $S$ may be
described in terms of certain statistics on~$B_n$. For $w\in B_n$,
define
\begin{align*}
\coxinv(w) & =\#\{(i,j)\in[n]^{2}\mid i<j, \; w(i)>w(j)\},\\
\coxneg(w) & =\#\{i\in[n]\mid w(i)<0\},\\
\coxnsp(w) & =\#\{\{i,j\}\subseteq[n]\mid i\neq j, \; w(i)+w(j)<0\}.
\end{align*}
It is well known (see~\cite[Proposition 8.1.1]{BjoernerBrenti/05})
that
\begin{equation}\label{equ:length}l(w)=\coxinv(w)+\coxneg(w)+\coxnsp(w).
\end{equation}
The descent set $D(w)$ of an element $w\in B_n$ may be characterised
as follows: $$D(w) = \{i\in[n-1]_0 \mid w(i) > w(i+1)\}.$$

We identify the parabolic subgroup $(B_n)_{[n-1]} = \{w\in B_n \mid
\coxneg(w)=0\}$ with the symmetric group $S_n$, with standard Coxeter
generating set $\{s_1,\dots,s_n\}$. In the combinatorial description
given above, this identifies $S_n$ with the group of $n\times n$
permutation matrices. We will freely switch between viewing elements
of $B_n$ as permutations of $[\pm n]_0$ or as signed permutation
matrices, as appropriate. Given a Coxeter group $W$ with Coxeter
generating set $S$, we usually just write $l$ for the associated
Coxeter length function. Only in case of ambiguity will we use a
subscript to indicate the relevant Coxeter group.

Let $M\in\Mat(r \times s;\Z)$. If $M$ has exactly one non-zero entry
in column $j\in[s]$ we write
\[
i_M(j):=i(j)\in[r]%:=|w(j)|
\]
for the unique integer $i$ such that $M_{ij}\neq0$; informally, $i(j)$
indicates the row of $M$ which contains the non-zero entry in
column~$j$. Similarly, if $M$ has exactly one non-zero entry in row $i\in[r]$ we write
\[
j_M(i):=j(i)\in[s]%:=|w^{-1}(i)|
\]
for the number of the column of $M$ which contains the non-zero entry
in row~$i$. In particular, if $w\in B_n$ then
$i_w(j) = |w(j)|$ and $j_w(i)=|w^{-1}(i)|$.

We call elements of the quotient $B_{n}^{[n-1]}$ \emph{ascending}. An
element $w\in B_n$ is ascending if and only if $w(1)<w(2) < \dots <
w(n)$. Such an element is determined by its \emph{row pattern}, that
is, by the function
\begin{equation}\label{def:row.pattern}
\rho_{w}:[n]\longrightarrow\{\pm1\},\qquad\rho_{w}(i)=w_{i,j(i)},
\end{equation}
defined for all $w\in B_n$.

Let $n\in\N$ and $I=\{i_{1},\dots,i_{l}\}_{<}\subseteq[n-1]_{0}$.  Our
first step towards proving Theorem~\ref{thm:Main} is to show that the
sum in \eqref{eq:def} is supported on relatively small and manageable
subsets of $B_{n}^{I^{\mathrm{c}}}$ which we now define.
\begin{defn} \label{def:chessboard} Set
\begin{align*}
\cbe & =\{(w_{ij})\in B_{n}\mid w_{ij}\neq0\Longrightarrow i+j\equiv0\bmod{(2)}\},\\
\cbo & =\{(w_{ij})\in B_{n}\mid w_{ij}\neq0\Longrightarrow i+j\equiv1\bmod{(2)}\},\\
\cb & =\cbe\cup\cbo.
\end{align*}
\end{defn}

We call $\cb$ the group of \emph{chessboard elements} and $\cbe$ the
subgroup of \emph{even chessboard elements}. Clearly $\cb$ contains
$\cbe$ as a subgroup of index~$2$.  The name comes from imagining a
signed permutation matrix $w\in B_{n}$ printed on an $n\times n$
``chessboard'' made up from white and black squares. The element $w$
is then a chessboard element exactly if all the non-zero entries of
$w$ occupy squares of the same colour.  Chessboard elements were
introduced in \cite{KlopschVoll/09} for the symmetric
group~$S_{n}$. Definition~\ref{def:chessboard} is an extension
of~\cite[Definition~5.3]{KlopschVoll/09} to the group~$B_{n}$.

Let $w=(w_{ij})\in\cbe$ and $m_1=\lfloor\frac{n+1}{2}\rfloor$,
$m_2=\lfloor\frac{n}{2}\rfloor$.  Let $w_{1}=(w_{2a+1,2b+1})$, where
$0\leq a,b\leq m_1$, and $w_{2}=(w_{2a,2b})$, where $1\leq a,b\leq m_2$.
Then $w_{1}\in B_{m_1}$ and $w_{2}\in B_{m_2}$. This defines a group
isomorphism
\[
\sigma_{0}:\cbe\longrightarrow B_{m_1}\times B_{m_2},\quad
w\longmapsto(w_{1},w_{2}).
\]
More generally, let $w\in B_{n}$ and define
\begin{alignat*}{2}
  w_{1} &
  = %(w_{ij})\in B_{m_1}, \quad & \text{ where } i&\in w(2[m]-1),\quad& j&\in 2[m]-1\\%
  (w_{i(2a-1),2a-1})\in B_{m_1},\quad&&1\leq a\leq m_1,\\
  w_{2} &
  = %(w_{ij})\in B_{m_1}, \quad & \text{ where } i&\in w(2[m_2]),\quad& j&\in 2[m_2]%
  (w_{i(2a),2a})\in B_{m_2},\quad&&1\leq a\leq m_2.
\end{alignat*}
Informally, $w_1$ is the submatrix of $w$ obtained by selecting the
odd-numbered columns of $w$ together with the corresponding rows of
$w$, and $w_2$ is obtained analogously by selecting the even-numbered
columns. We obtain a map of sets
\begin{equation}\label{def:sigma}
\sigma:B_{n}\longrightarrow B_{m_1}\times B_{m_2},\quad w\longmapsto(w_{1},w_{2}),
\end{equation}
whose restriction to $\cbe$ agrees with~$\sigma_{0}$.  Given $w_1\in
B_{m_1}$ and $w_2\in B_{m_2}$ we write
$w_1*w_2:=\sigma_0^{-1}(w_1,w_2)\in\cbe$ for the unique even
chessboard element in the fibre~$\sigma^{-1}(w_1,w_2)$.

Our next aim is to give a combinatorial description of the statistic
$L$, akin to the formula \eqref{equ:length} for the Coxeter length
function on $B_n$. To this end, we introduce the following
statistics.
\begin{defn}\label{def:abc}
  Let $r,s\in\N$ and $M=(M_{ij})\in\Mat(r\times s,\Z)$. Let
  $\mathcal{S}\subseteq[s]$ denote the set of indices of columns of
  $M$ which contain a unique non-zero entry. Define
\begin{align*}
  a(M) & =\#\{ j\in \mathcal{S} \mid M_{i(j),j}=-1,\;
  j\not\equiv0\bmod{(2)}\},\\
  b(M) & = \#\{(j,j')\in\mathcal{S}^2 \mid j<j',\; i(j)>i(j'),\; j\not\equiv j' \bmod{(2)}\},\\
c(M) & = \#\{(j,j')\in\mathcal{S}^2 \mid M_{i(j'),j'}=-1,\; j<j',\; i(j)<i(j'),\; j\not\equiv j' \bmod{(2)}\}.
\end{align*}
By viewing elements of $B_n$ as signed permutation matrices, these
formulae define, in particular, functions $a,b$ and $c$ on the
hyperoctahedral groups.
\end{defn}

\begin{exa}
Consider $w=[1,-4,-3,2]\in B_4$. Then $a(w)=c(w)=1$ and~$b(w)=2$.
\end{exa}

The following characterisation of $L$ will be used throughout the
paper.
\begin{lem}
  \label{lem:L}Let $w\in B_{n}$ and
  $\sigma(w)=(w_1,w_2)$.  Then
\begin{align}
  L(w) & =a(w)+b(w)+2c(w)\label{enu:L(1)}\text{ and }\\
  L(w) & =\coxneg(w_{1})+(\coxinv+\coxnsp)(w)-(\coxinv+\coxnsp)(w_{1})-(\coxinv+\coxnsp)(w_{2})\label{enu:L(2)}\\
  & =\coxneg(w_{1}) + l(w)-l(w_{1})-l(w_{2}).\nonumber
\end{align}
\end{lem}
\begin{proof}
  To prove \eqref{enu:L(1)}, set $$\mathfrak{N}_n = \{(i,j) \in [\pm
    n]_0^2 \mid |i| < |j|,\; i<j,\; i\not\equiv j \bmod{(2)}\}$$
  and $$ \mathfrak{M}_n(w) = \{(i,j) \in \mathfrak{N}_n \mid
  w(i)>w(j)\}.$$ By definition, $L(w) = \# \mathfrak{M}_n(w)$. Let
  $(i,j)\in\mathfrak{N}_n$.  If $|w(i)|>|w(j)|$ then exactly one of
  $(i,j)$ and $(-i,j)$ are in $\mathfrak{M}_n(w)$. (Note that in this
  case~$i \neq 0$.) Thus $b(w)=\#\{(i,j)\in\mathfrak{M}_n(w) \mid
  |w(i)|>|w(j)|\}$.  If $|w(i)|<|w(j)|$ we distinguish further by the
  sign of $w(j)$. If $w(j)<0$ then both $(i,j)$ and $(-i,j)$ are in
  $\mathfrak{M}_n(w)$. Note that $(i,j)=(-i,j)$ if and only if $i=0$,
  in which case $j$ is odd. If $w(j)>0$ then neither $(i,j)$ nor
  $(-i,j)$ are in $\mathfrak{M}_n(w)$. Thus
  $a(w)+2c(w)=\#\{(i,j)\in\mathfrak{M}_n(w) \mid |w(i)|<|w(j)|\}$ and
  $L(w) = a(w) + b(w) + 2c(w)$ as claimed.

We now prove \eqref{enu:L(2)}. First, it is clear that
$a(w)=\coxneg(w_{1})$. Omitting the parity conditions in the
definitions of the functions $b$ and $c$ given in
Definition~\ref{def:abc} yields
\begin{align*}
\bar{b}(M) & := \#\{(j,j')\in\mathcal{S}^2 \mid j<j',\; i(j)>i(j')\},\\
\bar{c}(M) & := \#\{(j,j')\in\mathcal{S}^2 \mid M_{i(j'),j'}=-1,\; j<j',\; i(j)<i(j')\}.
\end{align*}
We claim that $\bar{b}+2\bar{c}=\coxinv+\coxnsp$ on~$B_n$. To show
this, we make the following observations.  The function
$\bar{b}+2\bar{c}$ counts certain column pairs $(j,j')$, depending
only on the $2\times2$ submatrices determined by~$(j,j')$. The same is
true for the function~$\coxinv+\coxnsp$.  To establish the claim thus
amounts to checking it on~$B_{2}$. A simple calculation confirms it
there, and thus $\bar{b}+2\bar{c}=\coxinv+\coxnsp$ on~$B_n$. We
further observe that
$$
b(w) =\bar{b}(w)-\bar{b}(w_{1})-\bar{b}(w_{2})\quad \text{ and }\quad
c(w) =\bar{c}(w)-\bar{c}(w_{1})-\bar{c}(w_{2}).$$
Using \eqref{enu:L(1)} this yields
\begin{align*}
L(w) & =a(w)+b(w)+2c(w)\\ &
=a(w)+\bar{b}(w)-\bar{b}(w_{1})-\bar{b}(w_{2})+2(\bar{c}(w)-\bar{c}(w_{1})-\bar{c}(w_{2}))\\ &
=\coxneg(w_{1})+(\coxinv+\coxnsp)(w)-(\coxinv+\coxnsp)(w_{1})-(\coxinv+\coxnsp)(w_{2}).
\end{align*}
Using, finally, the facts that $l=\coxinv+\coxnsp+\coxneg$ (see
\eqref{equ:length}) and $\coxneg(w)=\coxneg(w_{1})+\coxneg(w_{2})$, we
obtain the second equality in \eqref{enu:L(2)}. \end{proof}

The unique longest element of $B_{n}$ is $w_{0}=[-1,-2,\dots,-n]$, of
length $l(w_{0})=n^{2}$. It is well known that the Coxeter length
function $l$ on $B_n$ is well-behaved under multiplication
by~$w_0$. More precisely, the equalities $$l(ww_0) = l(w_0w) = l(w_0)
- l(w)$$ hold for all $w\in B_n$;
cf.~\cite[Section~1.8]{Humphreys/90}. At least in this respect the
statistic $L$ behaves analogously.
\begin{cor}
  Let $w_{0}\in B_{n}$ be the longest element. Then, for all~$w\in
  B_{n}$, we have
\[
L(ww_{0})=L(w_{0}w)=L(w_{0})-L(w).
\]
Moreover, the trivial element in $B_n$ is the only element $w\in B_n$
with $L(w)=0$, and hence $w_{0}$ is the unique element in $B_{n}$ on
which $L$ attains its maximum~$\binom{n+1}{2}$.\end{cor}
\begin{proof}
  Let $w\in B_n$. Note that $w_{0}=-\Id_{n}$, where $\Id_{n}$ is the
  $n\times n$ identity matrix, so $ww_0=w_0w=-w$.  Obviously
  $\coxneg(w_0)=n$, and so
  $\coxneg(ww_{0})=n-\coxneg(w)=\coxneg(w_{0})-\coxneg(w)$.  Since
  $l=\coxinv+\coxnsp+\coxneg$, we thus have
\begin{equation}
(\coxinv+\coxnsp)(-w)=(\coxinv+\coxnsp)(w_{0})-(\coxinv+\coxnsp)(w)=n^{2}-n-(\coxinv+\coxnsp)(w).\label{eq:inv-nsp}
\end{equation}
Let $\sigma(w)=(w_1,w_2)\in B_{m_1} \times B_{m_2}$, where
$m_1=\lfloor\frac{n+1}{2}\rfloor$ and~$m_2=\lfloor\frac{n}{2}\rfloor$.
Using Lemma~\ref{lem:L}~\eqref{enu:L(2)} together with
\eqref{eq:inv-nsp} and the fact that $n=m_1+m_2$, we then obtain
\begin{align*}
L(ww_{0}) & =L(-w)\\
&=\coxneg(-w_{1})+(\coxinv+\coxnsp)(-w)-(\coxinv+\coxnsp)(-w_{1})-(\coxinv+\coxnsp)(-w_{2})\\
 & =m_1-\coxneg(w_{1})+n^{2}-n-(\coxinv+\coxnsp)(w)-(m_1^{2}-m_1-(\coxinv+\coxnsp)(w_{1}))\\
 & \quad-(m_2^{2}-m_2-(\coxinv+\coxnsp)(w_{2}))\\
 & =m_1+n^{2}-n-m_1^{2}+m_1-m_2^{2}+m_2-L(w) =m_1(2m_2+1)-L(w).
\end{align*}
Using Lemma~\ref{lem:L}~\eqref{enu:L(1)} it is easy to see that
$L(w_{0})=\binom{n+1}{2}$. Clearly~$m_1(2m_2+1)=\binom{n+1}{2}$, so
$L(ww_{0})=L(w_{0}w)=L(w_{0})-L(w)$, as asserted. This immediately
implies that $L(w_{0})=\binom{n+1}{2}$ is the maximal value attained
by~$L$. To see that $w_{0}$ is the unique element on which $L$ attains
its maximum, it suffices to show that $L(w)=0$ implies~$w=1$. Assume
thus that $L(w)=0$, for some $w\in B_{n}$. By
Lemma~\ref{lem:L}~\eqref{enu:L(1)}, this implies that
$a(w)=b(w)=c(w)=0$. Let $j\in[n-1]$. Then $b(w)=0$ implies that
$i(j)<i(j+1)$, and $c(w)=0$ then implies that $i(j+1)=i(j)+1$.  Since
this is true for all $j\in[n-1]$, we have either $w=1$ or
$w=s_{0}$. But $a(s_{0})=1$, so we must have $w=1$.\end{proof}

As mentioned previously, our approach to proving
Conjecture~\ref{conjecture} is to show that the sum in~\eqref{eq:def}
is supported on certain proper subsets
of~$B_{n}^{I^{\mathrm{c}}}$. The following is our first result in this
direction, and says that the sum is supported on the even chessboard
elements in~$B_{n}^{I^{\mathrm{c}}}$. Key to its proof is the
construction of a suitable sign reversing involution. For any subset
$X\subseteq B_{n}$ and $I\subseteq[n-1]_{0}$, we set
$$X^{I}:=X\cap
B_{n}^{I}.$$

\begin{lem}
  \label{lem:cb-supp}For $n\in\N$ and $I\subseteq[n-1]_{0}$,
\[
\sum_{w\in
  B_{n}^{I^{\mathrm{c}}}}(-1)^{l(w)}X^{L(w)}=\sum_{w\in\cbe^{I^{\mathrm{c}}}}(-1)^{l(w)}X^{L(w)}.
\]
\end{lem}
\begin{proof}
  Let $w=(w_{ij})\in B_{n}\setminus\cbe$. Thus there exists
  $i\in[n-1]_0$ such that $ j(i)\equiv j(i+1)\bmod{(2)}.  $ Let $i$ be
  minimal with this property and set
  $w^{*}:=s_{i}w$. Lemma~\ref{lem:L}~\eqref{enu:L(1)}
  implies that $L(w)=L(w^{*})$. Moreover, $l(w)=l(w^{*})\pm1$. Since
  $|j(i)-j(i+1)|\geq2$, we have $D(w)=D(w^{*})$. Note that $(w^*)^*=w$
  and $w \neq w^*$. Every element $w\in
  B_{n}^{I^{\mathrm{c}}}\setminus\cbe$ may thus be paired up with a
  unique, distinct element $w^{*}\in B_{n}^{I^{\mathrm{c}}}\setminus\cbe$, such
  that $(-1)^{l(w)}L(w)+(-1)^{l(w^*)}L(w^{*})=0$. This implies the
  assertion.
\end{proof}

\section{The case $I=\{0\}$\label{sec:The-case1}}

In this section we prove Conjecture~\ref{conjecture} in the case where
$I=\{0\}$, that is Case~\eqref{1} of Theorem~\ref{thm:Main}. In this
case, the sum in~\eqref{eq:def} runs over~$B_{n}^{[n-1]}$, that is,
ascending matrices. Let $\tilde{n}:=2[\frac{n-1}{2}]+1$ be the largest
odd integer less than or equal to $n$. Then, by definition,
\[
f_{n,\{0\}}(X)=\frac{(\underline{n})!}{\prod_{\sigma=1}^{\lfloor
    n/2\rfloor}(\underline{2\sigma})}=(\underline{1})(\underline{3})\cdots(\underline{\tilde{n}}).
\]
\begin{prop}
\label{pro:Ascending}Conjecture~\ref{conjecture} holds for
$I=\{0\}$, that is
$$\sum_{w\in B_n^{[n-1]}} (-1)^{l(w)} X^{L(w)} = (\ul{1})(\ul{3})\dots(\ul{\tilde{n}}).$$
\end{prop}
\begin{proof}
  By Lemma~\ref{lem:cb-supp}, it is enough to prove the assertion
  where the sum runs over~$\cbe^{[n-1]}$, that is ascending even
  chessboard elements.  Assume first that $n$ is odd, and that
  Proposition~\ref{pro:Ascending} is true for $n$. Since $n$ is odd,
  we have $\tilde{n}=\widetilde{n+1}=n$.  In this case restriction of
  to $[\pm n]_0$ yields a one-to-one correspondence between elements
  of~$\mathcal{C}_{n+1,0}^{[n]}$ and elements
  of~$\cbe^{[n-1]}$. Indeed, if $w\in \mathcal{C}_{n+1,0}^{[n]}$ then
  $w_{n+1,n+1}=1$. Moreover, it is clear that $L$ and $l$ are
  preserved under this correspondence.  Hence
$$ \sum_{w\in\mathcal{C}_{n+1,0}^{[n]}}(-1)^{l(w)}X^{L(w)}
  =\sum_{w\in\cbe^{[n-1]}}(-1)^{l(w)}X^{L(w)}
  =(\underline{1})(\underline{3})\cdots(\underline{\tilde{n}})=(\underline{1})(\underline{3})\cdots(\underline{\widetilde{n+1}}).$$
  Hence, if Proposition~\ref{pro:Ascending} is true for all odd $n$
  then it is also true for all even $n$.

We now prove Proposition~\ref{pro:Ascending} for odd $n$ by induction
in steps of two. For $n=1$ we have $f_{1,\{0\}}(X)=1-X$, and
$B_{1}=\{1,-1\}=B_{1}^{\varnothing}$, so
\[
\sum_{w\in
  B_{1}^\varnothing}(-1)^{l(w)}X^{L(w)}=(-1)^{l(1)}X^{L(1)}+(-1)^{l(-1)}X^{L(-1)}=1-X.
\]
Assume now that $n$ is odd and that Proposition~\ref{pro:Ascending}
holds for~$n$. We show how every element in
$\mathcal{C}_{n+2,0}^{[n+1]}$ is obtained from one in $\cbe^{[n-1]}$
in exactly one of two ways. Let $w\in\cbe^{[n-1]}$.  Then we may
associate to $w$ two elements in $\mathcal{C}_{n+2,0}^{[n+1]}$,
namely

\[
w^{+}:=\begin{pmatrix}w\\
 & 1\\
 &  & 1
\end{pmatrix}\quad \text{ and } \quad w^{-}:=\begin{pmatrix} &  & w\\
 & -1\\
-1
\end{pmatrix}.
\]
We claim that all elements in $\mathcal{C}_{n+2,0}^{[n+1]}$ are of
this form. To see this, let $v\in\mathcal{C}_{n+2,0}^{[n+1]}$. If
$v_{n+2,j(n+2)}=1$ then $j(n+2)=n+2$, since $v$ is ascending. Deleting
row $n+2$ and column $n+2$ leaves an element
$v'\in\mathcal{C}_{n+1,0}^{[n]}$.  Since
$v\in\mathcal{C}_{n+2,0}^{[n+1]}$, we must have $v'_{n+1,n+1}=1$, and
so $v$ is of the form~$w^+$. If~$v_{n+2,j(n+2)}=-1$ then
$j(n+2)=1$, since $v$ is ascending. Deleting row $n+2$ and column $1$
leaves an element $v'\in \mathcal{C}_{n+1,1}^{[n]}$. Hence
$w_{n+1,2}=-1$ and $v$ is of the form~$w^-$.

By Lemma~\ref{lem:L} \eqref{enu:L(1)}, we have $L(w^{+})=L(w)$ and
$L(w^{-})=n+2+L(w)$. Moreover, $l(w^{+})=l(w)$, and
$$
l(w^{-}) =(\coxneg+\coxnsp)(w^{-})=2+\coxneg(w)+2n+1+\coxnsp(w)
\equiv 1 + l(w)\bmod{(2)}.
$$
By the induction hypothesis, we obtain
\begin{align*}
\sum_{w\in\mathcal{C}_{n+2,0}^{[n+1]}}(-1)^{l(w)}X^{L(w)} & =\sum_{w\in\cbe^{[n-1]}}\left((-1)^{l(w^{+})}X^{L(w^{+})}+(-1)^{l(w^{-})}X^{L(w^{-})}\right)\\
 & =\sum_{w\in\cbe^{[n-1]}}\left((-1)^{l(w)}X^{L(w)}+(-1)^{1+l(w)}X^{n+2+L(w)}\right)\\
 & =\sum_{w\in\cbe^{[n-1]}}(-1)^{l(w)}X^{L(w)}(1-X^{n+2})\\
 & =(\underline{1})(\underline{3})\cdots(\underline{n})(1-X^{n+2})=(\underline{1})(\underline{3})\cdots(\underline{n})(\underline{\widetilde{n+2}}).%\\
% & =f_{n,\{0\}}(X).
\end{align*}
%This proves Proposition~\ref{pro:Ascending}.
\end{proof}

We record, without further proof, a corollary of the proof of
Proposition~\ref{pro:Ascending} on the structure of ascending even
chessboard elements.

\begin{lem}\label{lem:ascending-struct}
Let $w\in \mathcal{C}_{n,0}^{[n-1]}$ be an ascending even chessboard
element and $j\in \lfloor\frac{n}{2}\rfloor$. Then $i(2j)-i(2j-1)$ is
odd. Furthermore, the following hold:
\begin{enumerate}
\item If $w(2j)>0$ then $i(2j)-i(2j-1)>0$ and $w^{-1}(\imin)<0$ for all
$\imin\in\N$ such that $i(2j-1) <\imin< i(2j)$. Moreover,
$w(2j-1)>0$ unless possibly if $i(2j-1)=1$.
\item If $w(2j)<0$ then $i(2j-1) - i(2j)=1$.%, that is, $w(2j-1)=w(2j)-1$.
\end{enumerate}
\end{lem}
Informally, an ascending even chessboard element is built up from pairs of
adjacent columns satisfying one of the following:
\begin{itemize}
\item[(1)] Both columns typically contain positive entries, ``sandwiching'' an
even number of consecutive rows of $w$, all containing negative entries.
\item[(2)] Both columns contain negative entries
  in adjacent rows of $w$.
\end{itemize}
In particular, ascending even chessboard elements have no \emph{odd
  sandwich} in the sense of Definition~\ref{def:Odd sandwich}.

\section{The case $I=[n-1]_{0}$\label{sec:The-case2}}

In this section we prove Conjecture~\ref{conjecture} in the case where
$I=[n-1]_{0}$, that is Case~\eqref{2} of Theorem~\ref{thm:Main}. In
this case, we have $B_{n}^{I^{\mathrm{c}}}=B_{n}$ and
$f_{n,[n-1]_0}(X)=\frac{(\underline{n})!}{\prod_{\sigma=1}^{0}(\underline{
2\sigma})}=(\underline{n})!$.
By Lemma~\ref{lem:cb-supp}, the sum defining $f_{n,[n-1]_0}(X)$ is
supported on even chessboard matrices, that is
\[
\sum_{w\in B_{n}}(-1)^{l(w)}X^{L(w)}
%=\sum_{w\in\cbe^{\varnothing}}(-1)^{l(w)}X^{L(w)}
=\sum_{w\in\cbe}(-1)^{l(w)}X^{L(w)}.
\]
We now show that the latter sum is supported on diagonal
elements. More precisely, let
\[
\mathcal{D}_{n}=\{(w_{ij})\in\cbe\mid w_{ij}=0\text{ if }i\neq j\}
\]
denote the subgroup of $\cbe$ consisting of diagonal elements.
\begin{lem}\label{lem:diag-supp}
\[
\sum_{w\in\cbe}(-1)^{l(w)}X^{L(w)}=\sum_{w\in
  \mathcal{D}_{n}}(-1)^{l(w)}X^{L(w)}.
\]
\end{lem}
\begin{proof}
  Observe that $w\in \cbe \setminus \mcDn$ if and only if there exists
  $i\in[n-2]_0$ such that either $j(i+1) < \min\{j(i),j(i+2)\}$ or
  $j(i+1) > \max \{(j(i),j(i+2)\}$. Let $w\in \cbe\setminus \mcDn$ and
  let $i$ be minimal with respect to this property. Define $w^\circ :=
  s_{i+1}s_i s_{i+1}w \in \cbe \setminus \mcDn.$ Informally, $w^\circ$
  is obtained from $w$ by interchanging rows $i$ and $i+2$ if $i$ is
  positive, and by changing the sign in row $2$ if~$i=0$. Clearly
  $l(w^\circ) \equiv l(w) + 1 \bmod (2)$. Using Lemma \ref{lem:L} it
  is easy to see that $L(w^\circ) = L(w)$. Note that $(w^\circ)^\circ
  = w$ and~$w\neq w^\circ$. Every element $w\in\cbe\setminus\mcDn$ may
  thus be paired up with a unique, distinct element $w^{\circ}\in \cbe
  \setminus \mcDn$ such that
  $(-1)^{l(w)}L(w)+(-1)^{l(w^\circ)}L(w^{\circ})=0$. This implies the
  assertion.

%We have
%  thus shown that for every element $w\in\cbe\setminus \mcDn$ there
%  exists a unique, distinct element $w^\circ\in \cbe \setminus \mcDn$
%  whose contribution to the sum on the left side cancels with the
%  contribution of~$w$. This proves the lemma.
\end{proof}

\begin{prop}
\label{pro:Perm-case}Conjecture~\ref{conjecture} holds for
$I=[n-1]_{0}$, that is
$$\sum_{w\in B_n} (-1)^{l(w)}X^{L(w)} = (\ul{n})!.$$\end{prop}
\begin{proof}
  The proof is by induction on~$n$. The assertion holds trivially
  for~$n=1$.  Assume now that the assertion is true for some
  $n-1\geq1$. Given $v\in \mathcal{D}_{n-1}$ we define
\[
v^{+}:=\begin{pmatrix}v\\
 & 1
\end{pmatrix}\in\mathcal{D}_{n} \quad \textrm{ and }\quad
v^{-}:=\begin{pmatrix}v\\
 & -1
\end{pmatrix}\in\mathcal{D}_{n}.
\]
Using the formula $l=\coxinv+\coxneg+\coxnsp$ (cf.~\eqref{equ:length})
and Lemma~\ref{lem:L}~\eqref{enu:L(1)} we see that
\begin{alignat*}{2}
l(v^{+})&=l(v), & \qquad l(v^{-})&=l(v)+2n-1,\\
L(v^{+})&=L(v), & \qquad L(v^{-})&=L(v)+n.
\end{alignat*}
Hence, by Lemma~\ref{lem:cb-supp}, Lemma~\ref{lem:diag-supp} and the
induction hypothesis, we obtain

\begin{align*}
\sum_{w\in B_{n}}(-1)^{l(w)}X^{L(w)} &
%=\sum_{w\in\cbe^{\varnothing}}(-1)^{l(w)}X^{L(w)}
=\sum_{w\in\cbe}(-1)^{l(w)}X^{L(w)} =\sum_{w\in
  \mathcal{D}_{n}}(-1)^{l(w)}X^{L(w)}\\ & =\sum_{v\in
\mathcal{D}_{n-1}}\left((-1)^{l(v^{+})}X^{L(v^{+})}+(-1)^{l(v^{-})}X^{L(v^{-})}
\right)\\ &
= \sum_{v\in
\mathcal{D}_{n-1}}\left((-1)^{l(v)}X^{L(v)}+(-1)^{l(v)+2n-1}X^{L(v)+n}\right)\\
&
=\sum_{v\in \mathcal{D}_{n-1}}(-1)^{l(v)}X^{L(v)}(1-X^{n})
=(\underline{n-1})!(1-X^{n})=(\underline{n})!.
\end{align*}

\end{proof}

\section{The case $n$ even and $I$ even\label{sec:The-case3}}\label{sec:case3}

In this section we push further the ideas that led to the proof of
Lemma~\ref{lem:cb-supp}. There we proved that the relevant sums over
$B_n^{I^{\mathrm{c}}}$ are supported over chessboard matrices
$\mathcal{C}_{n,0}^{I^{\mathrm{c}}}$. In the proof we described a sign
reversing involution $*$ on $B_n\setminus\mathcal{C}_{n,0}$ such that
$D(w) = D(w^*)$, $L(w) = L(w^*)$ and $l(w)\not\equiv l(w^*) \bmod (2)$
for all $w\in B_n\setminus\mathcal{C}_{n,0}$. Consequently, these
elements' contributions to the sums in question cancelled each other
out. A similar idea was used in the proof of
Lemma~\ref{lem:diag-supp}. In the current section we further restrict
the ``supporting sets'' $\mathcal{C}_{n,0}^{I^{\mathrm{c}}}$ and show
how, under suitable conditions, elements outside these sets may be
cancelled by means of a sign reversing involution; see
Definition~\ref{def:w.vee}. In Sections~\ref{subsec:first} and
\ref{subsec:second} we establish additivity results for $L$ with
respect to two parabolic factorisations. In conjunction, they allow us
to establish Conjecture~\ref{conjecture} in the case where
\foreignlanguage{english}{$n$ is even and
  $I\subseteq[n-1]_{0}\cap2\Z$}, that is Case~\eqref{3} of
Theorem~\ref{thm:Main}, in Proposition~\ref{pro:Even}.

\subsection{Parabolic factorisations and supporting
sets}\label{subsec:parabolic} Recall that we may factorise any element
$w\in B_{n}$ as $w=w^{[n-1]}w_{[n-1]}$, where $w^{[n-1]}\in
B_{n}^{[n-1]}$ is ascending, and $w_{[n-1]}\in\langle
s_1,\dots,s_{n-1}\rangle\cong S_{n}$; cf.~\eqref{equ:factorisation}.
Let $w\in\cbe$ be an even chessboard element. Since $\cb$ is a group
containing $\cbe$ as a subgroup, there are three possibilities for
this factorisation of $w$:
\begin{enumerate}
\item $w^{[n-1]},w_{[n-1]}\in\cbe$,
\item $w^{[n-1]},w_{[n-1]}\in\cbo$,
\item $w^{[n-1]},w_{[n-1]}\in B_{n}\setminus\cb$.
\end{enumerate}

\begin{defn}
Let
\[
\firste=\{w\in\cbe\mid w^{[n-1]},w_{[n-1]}\in\cbe\}
\]
 denote the set of even chessboard elements whose factorisation is
into even chessboard elements. Similarly, let
\[
\firstcb=\{w\in\cbe\mid(w^{[n-1]},w_{[n-1]}\in\cbe)\vee(w^{[n-1]},w_{[n-1]}
\in\cbo)\}
\]
 denote the set of even chessboard elements whose factorisation is
into chessboard elements.
\end{defn}
Note that $\firste\subseteq\firstcb\subseteq\mathcal{C}_{n,0}$ and
that $\firstcb=\firste$ if $n$ is odd.

Some of the key features of the case where $n$ and $I$ are even are
recorded in the following lemma. A subset
$I=\{i_{1},i_{2},\dots,i_{l}\}_{<}\subseteq[n-1]_{0}\cap2\Z$ is called
\emph{even}. We say that $w$ is of \emph{even descent type} if
$I=D(w)$ is even.

\begin{lem} \label{lem:evenperm-block}Let $n$ be even and $w\in\cb\cap
  S_{n}$ be a chessboard element in $S_{n}$.  Suppose that $D(w)$ is
  even, and write $\sigma(w)=(w_{1},w_{2})\in S_{n/2}\times
  S_{n/2}$. Then $w\in\cbe$ and $w_{1}=w_{2}\in
  %\mathcal{C}_{n/2,0}\cap
  S_{n/2}^{(I/2)^{\mathrm{c}}}$. In
  particular, for all even $I\subseteq[n-1]_0$,
\[
\firstcb^{I^{\mathrm{c}}}=\firste^{I^{\mathrm{c}}}.
\]
Moreover, $l_{S_{n}}(w)=4l_{S_{n/2}}(w_{1})$.\end{lem} Informally,
Lemma~\ref{lem:evenperm-block} states that $w$ is a {}``block
permutation matrix'', composed of $2\times2$ identity matrices.

\begin{proof} Let $w=(w_{ij})$, and recall that $w_{i,j(i)}$ denotes
  the non-zero entry in the $i$-th row. Assume first that
  $w\in\cbo$.  Then $i+j(i)$ is odd for all~$i\in[n]$, so
  $j(1)$ is even. This implies that $w$ has a descent at $j(1)-1$,
  which is impossible since $D(w)$ is even. Thus $w\in\cbe$,
  and $i+j(i)$ is even for all~$i\in[n]$. Suppose that $j(2)\leq
  j(1)-1$ or $j(2)\geq j(1)+3$. Then $i(j(2)-1)\geq3$, so there is a
  descent at $j(2)-1$, contradicting the assumption that $D(w)$ is
  even. Thus $j(2)=j(1)+1$. Continuing the same argument for the
  $2i+1$-th and $2i+2$-th row, for each $i\in[\frac{n}{2}-1]$, we
  obtain
\[
j(2i+2)=j(2i+1)+1,\text{ for all }i\in[\frac{n}{2}-1]_0.
\]
By definition~\eqref{def:sigma} of the map $\sigma$ this means
that~$w_{1}=w_{2}\in S_{n/2}$.%\in\mathcal{C}_{n/2,0}$.

Furthermore, $w$ has a descent at $2a$ if and only if $w_{1}$ has a
descent at $a$, for all~$a\in[n-1]$. Hence $w_{1}\in
S_{n/2}^{(I/2)^{\mathrm{c}}}$.  This implies that $w\in\firstcb$ if
and only if $w\in\firste$. Indeed, if $w_{[n-1]}\in\cb$, then we have
shown that $w_{[n-1]}\in\cbe$, and so $w^{[n-1]}\in\cbe$, and hence
$w\in\firste$. Thus~$\firstcb^{I^{\mathrm{c}}}=\firste^{I^{\mathrm{c}}}$
for all even~$I$.  The statement about the lengths is
clear. \end{proof}

\begin{defn}
\label{def:Odd sandwich}Let $w=(w_{ij})\in B_{n}$. A pair of natural
numbers $(r,h)$, where $r\in[n-2]$ and $h\in[n-1]$ is odd, is said to be an
\emph{odd sandwich in }$w$ if it satisfies one of the following
conditions: \begin{enumerate}

\item \label{Odd-sand-proper}$w_{r,j(r)}=w_{r+h+1,j(r+h+1)}$, and
$w_{r,j(r)}\neq w_{r+i,j(r+i)}$ for all $i\in[h]$,

\item \label{Odd-sand-degen}$r=1$, $w_{1,j(1)}=w_{1+i,j(1+i)}$ for all
  $i\in[h]$, and $w_{1,j(1)}\neq w_{1+h+1,j(1+h+1)}$. \end{enumerate}
We say that $w$ \emph{has an odd sandwich} if there exists an odd
sandwich in $w$. \end{defn} Recall that $s_{0}\in S$ is the Coxeter
generator such that, for any $w\in B_{n}$, the matrix $s_{0}w$ is
obtained by changing the sign of~$w_{1,j(1)}$. Informally speaking,
$w$ has an odd sandwich if and only if in either $w$ or $s_{0}w$ there
exists a row containing a~$1$, followed by an odd number of
consecutive rows containing~$-1$s, followed by a row containing a $1$,
or if there exists a row containing a~$-1$, followed by an odd number
of consecutive rows containing~$1$s, followed by a row containing a
$-1$.

\begin{lem} \label{lem:odd-sandwich} Let $w\in\cbe$. Then
  $w\in\mathcal{M}_{n}$ if and only if $w$ has no odd
  sandwich.\end{lem}

\begin{proof} Write $w=(w_{ij})=w^{[n-1]}w_{[n-1]}$.  Then $w$ has the
  same row pattern as $w^{[n-1]}$;
  cf.~\eqref{def:row.pattern}. Moreover, $w\in\firstcb$ if and only if
  $w^{[n-1]}\in\cb$.  To prove the lemma, it therefore suffices to
  prove that for any $v\in B_n^{[n-1]}$ we have $v\in\cb^{[n-1]}$ if
  and only if $v$ has no odd sandwich.

  Assume that $v\in\cb^{[n-1]}$ and that $v$ has an odd
  sandwich~$(r,h)$.  It is easily seen that the smallest integer
  $i\in[n-1]$ such that $v_{i,j_v(i)}\neq v_{i+1,j_v(i+1)}$ is
  odd. Since $v$ is ascending, we have $|j_v(r)-j_v(r+h+1)|=1$. But
  since $h$ is odd, we have $ r+j_v(r)\not\equiv
  r+h+1+j_v(r+h+1)\bmod{(2)}, $ and so $v\not\in\cb$;
  contradiction. Thus $v\in\mathcal{C}_n^{n-1}$ implies that $v$ does
  not have an odd sandwich.

  Conversely, assume that $v\in B_n^{[n-1]}\setminus\cb$. This means
  that there exists an integer $j\in[n]$, such that
  $i_v(j)+j\not\equiv i_v(j+1)+j+1\bmod{(2)}$. (Informally, the
  non-zero entries in columns $j$ and $j+1$ are on chessboard squares
  of different colours.) In particular, $h:=|i_v(j)-i_v(j+1)|-1$ is
  odd. Let $r:=\min\{i_v(j),i_v(j+1)\}$, so that
  $r+h+1=\max\{i_v(j),i_v(j+1)\}$.  If $v_{r,j(r)}=v_{r+h+1,j(r+h+1)}$
  then, because $v$ is ascending, $v_{r,j(r)}\neq v_{r+s,j(r+s)}$, for
  all $s\in[h]$, so $v$ has an odd sandwich~$(r,h)$. If
  $v_{r,j(r)}\neq v_{r+h+1,j(r+h+1)}$ then, again because $v$ is
  ascending, $r=1$, and so $v$ has an odd sandwich~$(1,h)$. In either
  case $v\in B_n^{[n-1]}\setminus\cb$ implies that $v$ has an odd
  sandwich.
\end{proof}

Let $w\in\cbe\setminus\firstcb$. By Lemma~\ref{lem:odd-sandwich} this
means that~$w$ has an odd sandwich. Let $(r,h)$ be the \emph{topmost
  odd sandwich} in~$w$, that is the unique odd sandwich in $w$ such
that if $(r',h')$ is another odd sandwich in $w$, then $r\leq r'$.  In
the following we define an element $w^{\vee}\in
\cbe\setminus\mathcal{M}_n$ with the property that $L(w)=L(w^{\vee})$,
the positive parts of the descent sets $D(w)$ and $D(w^{\vee})$ agree
and the parities of $l(w)$ and $l(w^{\vee})$ differ.  For this end, we
factorise $w=w^{[n-1]}w_{[n-1]}$ with $w^{[n-1]}\in B_n^{[n-1]}$ and
$w_{[n-1]}\in (B_{n})_{[n-1]}\cong S_n$. Since $w^{[n-1]}$ is
ascending, the non-zero entries in rows $r$ and $r+h+1$ must lie in
adjacent columns; in other words, if $j:=j_{w^{[n-1]}}(r)$ and
$j':=j_{w^{[n-1]}}(r+h+1)$ then
$|j-j'|=1$. Set~$\mu=\min\{j,j'\}\in[n]$.

\begin{defn}\label{def:w.vee}
  Given $w\in\cbe\setminus\firstcb$ with topmost odd sandwich $(r,h)$
  and $\mu$ as above. Set
\[
w^{\vee}=w^{[n-1]}s_{\mu}w_{[n-1]}\in\cbe\setminus\firstcb.
\]
\end{defn}

Informally, $w^{\vee}$ is obtained from $w=w^{[n-1]}w_{[n-1]}$ from
transposing columns $\mu$ and $\mu+1$ in $w^{[n-1]}$ or, equivalently,
transposing rows $\mu$ and $\mu+1$ in $w_{[n-1]}$.  The element
$w^{\vee}$ may also be thought of as obtained from $w$ by
interchanging columns $r$ and~$r+h+1$, deliminating the topmost odd
sandwich in~$w$. Before we prove that the involution $w\mapsto
w^{\vee}$ on $\mathcal{C}_{n,0}\setminus\mathcal{M}_n$ has the desired
properties, we consider an example.

\begin{exa}\label{exa}
  For $n=3$, let
$$w = \left( \begin{matrix} &&-1\\&-1&\\1&&\end{matrix}\right) =
\left( \begin{matrix} &-1&\\-1&&\\&&1\end{matrix}\right)
  \left( \begin{matrix} &1&\\&&1\\1&&\end{matrix}\right) =
  w^{[2]}w_{[2]}\in \mathcal{C}_{3,0}^{\{0,2\}} \setminus
  \mathcal{M}_3,$$ with $l(w) = 5$, $L(w) = 3$ and $D(w) = \{1\}$. The
  unique -- and therefore topmost -- odd sandwich in $w$ is
  $(r,h)=(1,1)$, involving the first and last row. Clearly
  $\mu=\min\{2,3\}=2$, and thus
$$w^{\vee} = \left( \begin{matrix} &-1&\\-1&&\\&&1\end{matrix}\right)
    s_2 \left( \begin{matrix} &1&\\&&1\\1&&\end{matrix}\right) =
    \left( \begin{matrix} -1&&\\&-1&\\&&1\end{matrix}\right),$$ with
      $l(w^{\vee}) = 4$, $L(w^{\vee}) = 3$ and $D(w^{\vee})=\{0,1\}$.
\end{exa}

\begin{lem}\label{lem:Anti-particles}Let
  $w=(w_{ij})\in\cbe\setminus\mathcal{M}_{n}$ with topmost odd
  sandwich $(r,h)$. Then
  $D(w)\setminus\{0\}=D(w^{\vee})\setminus\{0\}$, and if $(r,h)$
  satisfies Definition~\ref{def:Odd sandwich}~\eqref{Odd-sand-proper}
  then $D(w)=D(w^{\vee})$. Moreover,
\[
L(w)=L(w^{\vee}) \quad\text{and}\quad l(w) = l(w^{\vee})\pm1.
\]
 \end{lem} \begin{proof} We first prove the statements about the descent
types. Let $w^{\vee}=w^{[n-1]}s_{\mu}w_{[n-1]}$ as above. Since
$w_{[n-1]}^{-1}(\mu)\equiv w_{[n-1]}^{-1}(\mu+1)\bmod{(2)}$ the
non-zero entries of $w_{[n-1]}$ in rows $\mu$ and $\mu+1$ are not in
adjacent columns. Thus, transposing rows $\mu$ and $\mu+1$ does not
change the descent type of $w_{[n-1]}$, that is,
$D(w_{[n-1]})=D(s_{\mu}w_{[n-1]})$.  For any element $u\in B_{n}$, we
have $D(u)\setminus\{0\}=D(u_{[n-1]})$.  Since
$(w^{\vee})_{[n-1]}=s_{\mu}w_{[n-1]}$, we get
$D(w)\setminus\{0\}=D(w_{[n-1]})=D(s_{\mu}w_{[n-1]})=D(w^{\vee})\setminus\{0\}$.
If $(r,h)$ satisfies Definition~\ref{def:Odd
  sandwich}~\eqref{Odd-sand-proper} then
$w_{r,j(r)}=w_{r+h+1,j(r+h+1)}$ so $0\in D(w)$ if and only if $0\in
D(w^{\vee})$ and thus $D(w)=D(w^{\vee})$.

We now prove that $L(w)=L(w^{\vee})$. Let
$j_{\text{min}}:=\min\{j(r),j(r+h+1)\}$ and
$j_{\text{max}}:=\max\{j(r),j(r+h+1)\}$. Then $j_{\text{min}}\equiv
j_{\text{max}}\bmod{(2)}$, since $w\in\cbe$ and
$|i(j_{\text{max}})-i(j_{\text{min}})|=h+1$ is even. We write
$w=(w_{ij})$ and $w^\vee=(w^\vee_{ij})$, where $i,j,\in[n]$. Consider
the $(h+2)\times n$-submatrices
\begin{equation*}
v :=(w_{ij}) \textrm{ and } v^{\vee} :=(w_{ij}^{\vee}), \textrm{ where }r \leq i
\leq r+h+1,\quad j\in[n].
\end{equation*}
Recall that $v$ and $v^{\vee}$ are obtained from one another by
interchanging their first and last rows.  Using the fact that $L =
a+b+2c$ (cf.~Lemma~\ref{lem:L}~\eqref{enu:L(1)}), and noting that $w$
and $w^\vee$ coincide outside of the rows $i$ such that $r\leq i \leq
r+h+1$, we see that in order to prove that $L(w)=L(w^{\vee})$, it is
sufficient to prove that $(a+b+2c)(v)=(a+b+2c)(v^{\vee})$. To prove
the latter it is sufficient to show that for any column $j$ in $v$,
the contribution to $L$ from the three columns $j$,$j_{\text{min}}$,
\foreignlanguage{english}{$j_{\text{max}}$ in $v$ is equal to the
}contribution to $L$ from the three columns $j$,$j_{\text{min}}$,
\foreignlanguage{english}{$j_{\text{max}}$ in $v^{\vee}$. As
  $L=a+b+2c$ and $a(v)=a(v^{\vee})$, it is enough to consider the
  contribution to $b$ and $c$. Let $j$ be a non-zero column in $v$,
  that is $j\in[n]$ such that~$r\leq i(j)\leq r+h+1$. }Since
\foreignlanguage{english}{we only need to consider the contribution to
  $b$ and $c$ from the }columns $j$,$j_{\text{min}}$,
\foreignlanguage{english}{$j_{\text{max}}$, we may assume that
  $j\not\equiv j_{\mathrm{min}}\bmod{(2)}$, which is equivalent to
  $j\not\equiv j_{\mathrm{max}}\bmod{(2)}$, since
  $j_{\mathrm{min}}\equiv j_{\mathrm{max}}\bmod{(2)}$. There are then
  three possible cases: }$j<j_{\mathrm{min}}$,
$j_{\mathrm{min}}<j<j_{\mathrm{max}}$, and $j_{\mathrm{max}}<j$,
respectively. In the sequel we consider only the case that $w_{r,j(r)}
= w_{r+h+1,j(r+h+1)}$ (cf.\ Definition~\ref{def:Odd
  sandwich}~\eqref{Odd-sand-proper}), omitting similar arguments for
the case that $w_{r,j(r)}\neq w_{r+h+1,j(r+h+1)}$ (cf.\
Definition~\ref{def:Odd sandwich}~\eqref{Odd-sand-degen}).

Consider the first case, $j<j_{\mathrm{min}}$. Suppose that
$w_{i(j_{\mathrm{min}}),j_{\mathrm{min}}}=w_{i(j_{\mathrm{max}}),j_{\mathrm{max}
}}=1$.
If $i(j_{\mathrm{min}})<i(j_{\mathrm{max}})$, the total contribution
to $b+2c$ from the column pairs $(j,j_{\mathrm{min}})$ and
$(j,j_{\mathrm{max}})$ is $1$. If
$i(j_{\mathrm{min}})>i(j_{\mathrm{max}})$, the total contribution to
$b+2c$ from the column pairs $(j,j_{\mathrm{min}})$ and
$(j,j_{\mathrm{max}})$ is also $1$. Suppose on the other hand that
$w_{i(j_{\mathrm{min}}),j_{\mathrm{min}}}=w_{i(j_{\mathrm{max}}),j_{\mathrm{max}
}}=-1$.
If $i(j_{\mathrm{min}})<i(j_{\mathrm{max}})$, the total contribution
to $b+2c$ from the column pairs $(j,j_{\mathrm{min}})$ and
$(j,j_{\mathrm{max}})$ is $3$. If
$i(j_{\mathrm{min}})>i(j_{\mathrm{max}})$, the total contribution to
$b+2c$ from the column pairs $(j,j_{\mathrm{min}})$ and
$(j,j_{\mathrm{max}})$ is also $3$. Thus $(a+b+2c)(v)=(a+b+2c)(v^{\vee})$ in
the first case.\par Next, consider the second case,
$j_{\mathrm{min}}<j<j_{\mathrm{max}}$.  Suppose that
$w_{i(j_{\mathrm{min}}),j_{\mathrm{min}}}=w_{i(j_{\mathrm{max}}),j_{\mathrm{max}
}}=1$.
If $i(j_{\mathrm{min}})<i(j_{\mathrm{max}})$, the total contribution
to $b+2c$ from the column pairs $(j_{\mathrm{min}},j)$ and
$(j,j_{\mathrm{max}})$ is $2$. If
$i(j_{\mathrm{min}})>i(j_{\mathrm{max}})$, the total contribution to
$b+2c$ from the column pairs $(j_{\mathrm{min}},j)$ and
$(j,j_{\mathrm{max}})$ is also $2$. Suppose on the other hand that
$w_{i(j_{\mathrm{min}}),j_{\mathrm{min}}}=w_{i(j_{\mathrm{max}}),j_{\mathrm{max}
}}=-1$.
If $i(j_{\mathrm{min}})<i(j_{\mathrm{max}})$, the total contribution
to $b+2c$ from the column pairs $(j_{\mathrm{min}},j)$ and
$(j,j_{\mathrm{max}})$ is $2$. If
$i(j_{\mathrm{min}})>i(j_{\mathrm{max}})$, the total contribution to
$b+2c$ from the column pairs $(j_{\mathrm{min}},j)$ and
$(j,j_{\mathrm{max}})$ is also $2$. Thus $(a+b+2c)(v)=(a+b+2c)(v^{\vee})$ in the
second case.\par Finally, consider the third case,
$j_{\mathrm{max}}<j$.  Suppose that
$w_{i(j_{\mathrm{min}}),j_{\mathrm{min}}}=w_{i(j_{\mathrm{max}}),j_{\mathrm{max}
}}=1$.
If $i(j_{\mathrm{min}})<i(j_{\mathrm{max}})$, the total contribution
to $b+2c$ from the column pairs $(j_{\mathrm{min}},j)$ and
$(j_{\mathrm{max}},j)$ is $1$. If
$i(j_{\mathrm{min}})>i(j_{\mathrm{max}})$, the total contribution to
$b+2c$ from the column pairs $(j_{\mathrm{min}},j)$ and
$(j_{\mathrm{max}},j)$ is also $1$. Suppose on the other hand that
$w_{i(j_{\mathrm{min}}),j_{\mathrm{min}}}=w_{i(j_{\mathrm{max}}),j_{\mathrm{max}
}}=-1$.
If $i(j_{\mathrm{min}})<i(j_{\mathrm{max}})$, the total contribution
to $b+2c$ from the column pairs $(j_{\mathrm{min}},j)$ and
$(j_{\mathrm{max}},j)$ is $2$. If
$i(j_{\mathrm{min}})<i(j_{\mathrm{max}})$, the total contribution to
$b+2c$ from the column pairs $(j_{\mathrm{min}},j)$ and
$(j_{\mathrm{max}},j)$ is $1$. If
$i(j_{\mathrm{min}})>i(j_{\mathrm{max}})$, the total contribution to
$b+2c$ from the column pairs $(j_{\mathrm{min}},j)$ and
$(j_{\mathrm{max}},j)$ is also $2$. Thus $(a+b+2c)(v)=(a+b+2c)(v^{\vee})$ in the
third case.\par

%We have shown that $L(v)=L(v^{\vee})$, and
%hence $L(w)=L(w^{\vee})$, whenever
%$w_{i(j_{\mathrm{min}}),j_{\mathrm{min}}=w_}}}$,
%that is, when the odd sandwich $(r,h)$ is such that
%$w_{r,j(r)}=w_{r+1,j(r+1)}$%.
%The case when the odd sandwich is such that $w_{r,j(r)}\neq w_{r+1,j(r+1)}$
%is treated in a similar way.\par

To finish the proof of the lemma, recall from~\eqref{equ:add.length}
that for any $g\in B_{n}$, and any $s\in S$, we have
$l(g)=l(g^{[n-1]})+l(g_{[n-1]})$, and $l(sg)=l(g)\pm1$. Thus
$$l(w^{\vee})=l(w^{[n-1]})+l(s_{\mu}w_{[n-1]})=l(w^{[n-1]})+l(w_{[n-1]}
)\pm1=l(w)\pm1.$$\end{proof}

\begin{cor} \label{cor:Monochrome-supp} Let $n\in\N$ and
  $I\subseteq[n-1]_{0}$. Then
\begin{equation}\label{equ:monochrome}
\sum_{w\in B_{n}^{(I_{0})^{\mathrm{c}}}}(-1)^{l(w)}X^{L(w)}=\sum_{w\in\firstcb^{
(I_{0})^{ \mathrm{c}}}}(-1)^{l(w)}X^{L(w)}.
\end{equation}
Assume that either $n$ is odd or both $n$ and $I\subseteq[n-1]_{0}$
are even. Then
\[
\sum_{w\in B_{n}^{(I_0)^{\mathrm{c}}}}(-1)^{l(w)}X^{L(w)}=\sum_{w\in\firste^{
(I_0)^{\mathrm { c } }}}(-1)^{l(w)}X^{L(w)}.
\]
 \end{cor}

 \begin{proof} Without loss of generality we may assume that $0\in I$,
   so that $I=I_0$. By Lemma~\ref{lem:cb-supp} the sum over
   $B_{n}^{I^{\mathrm{c}}}$ is supported
   on~$\cbe^{I^{\mathrm{c}}}$. Lemma~\ref{lem:Anti-particles} asserts
   that for every $w\in\cbe^{I^{\mathrm{c}}}\setminus\firstcb$ there
   exists a unique $w^{\vee}\in\cbe^{I^{\mathrm{c}}}\setminus\firstcb$
   such that
   $(-1)^{l(w)}X^{L(w)}+(-1)^{l(w^{\vee})}X^{L(w^{\vee})}=0$. Moreover,
   $D(w)\setminus\{0\}=D(w^{\vee})\setminus\{0\}$, so $w\in
   B_{n}^{I^{\mathrm{c}}}$ if and only if $w^{\vee}\in
   B_{n}^{I^{\mathrm{c}}}$.  Hence the sum over
   $B_{n}^{I^{\mathrm{c}}}$ is supported on
   $\firstcb^{I^{\mathrm{c}}}$.

   When $n$ is even and $I\subseteq[n-1]_{0}$ is even,
   Lemma~\ref{lem:evenperm-block} states that
   $\firstcb^{I^{\mathrm{c}}}=\firste^{I^{\mathrm{c}}}$, whence the
   second equality. When $n$ is odd, it follows from the first,
   as~$\firstcb=\firste$.
\end{proof}

\begin{rem}
Example~\ref{exa} illustrates that the sign reversing involution
$\vee$ on $\mathcal{C}_{n,0}\setminus\mathcal{M}_n$ does not, in
general, preserve the descent type. This is in contrast to the
involution~$*$ defined in the proof of Lemma~\ref{lem:cb-supp}. The
weaker statement~\eqref{equ:monochrome} is, however, sufficient for
our application in the proof of Proposition~\ref{pro:Even}.
\end{rem}

\subsection{A first additivity result for $L$}\label{subsec:first}

We now consider how the statistic $L$ behaves with respect to the
parabolic factorisation $w=w^{[n-1]}w_{[n-1]}$. For an arbitrary
element $w\in\cbe$, it is not necessarily true that $L$ is additive
with respect to this factorisation, that
is~$L(w)=L(w^{[n-1]})+L(w_{[n-1]})$.  A counter-example is given by
\[
w=\begin{pmatrix}1\\
 & -1
\end{pmatrix}=\begin{pmatrix} & 1\\
-1
\end{pmatrix}\begin{pmatrix} & 1\\
1
\end{pmatrix}=w^{[1]}w_{[1]}\in\mathcal{M}_{2}\setminus \mathcal{E}_2,
\]
 where $L(w)=2$, $L(w^{[1]})=2$ and~$L(w_{[1]})=1$.

 The following result shows that the situation improves when we assume
 that~$w\in\firste$.
\begin{prop} \label{prop:first}Suppose that $w\in\firste$. Then
\[
L(w)=L(w^{[n-1]})+L(w_{[n-1]}).
\]
 \end{prop}

 \begin{proof} Since $w\in\firste$, we have
   $w^{[n-1]},w_{[n-1]}\in\cbe$.  Let $w=w_{1}*w_{2}$ and
   $w^{[n-1]}=(w^{[n-1]})_{1}*(w^{[n-1]})_{2}$. We claim that
\begin{equation}
w^{[n-1]}=(w_{1})^{[n-1]}*(w_{2})^{[n-1]},\label{eq:(ascending)}
\end{equation}
that is $(w^{[n-1]})_{1}=(w_{1})^{[n-1]}$ and
$(w^{[n-1]})_{2}=(w_{2})^{[n-1]}$.  Indeed, the ascending matrix
$w^{[n-1]}$ is obtained from $w$ by a permutation of columns, and
since both $w^{[n-1]}$ and $w$ are chessboard elements, each column of
$w$ is moved an even amount to obtain the corresponding column
of~$w^{[n-1]}$. Clearly, $w^{[n-1]}$ has the same row pattern as $w$;
cf.~\eqref{def:row.pattern}.  For any $v\in\cbe$, with
$v=v_{1}*v_{2}$, every non-zero entry at $(i,j)$ is either equal to
the entry at $(\frac{i+1}{2},\frac{j+1}{2})$ in $v_{1}$, if $i$ (and
therefore $j$) is odd, or is equal to the entry at
$(\frac{i}{2},\frac{j}{2})$ in $v_{2}$, if $i$ (and therefore $j$) is
even. The row pattern of $(w^{[n-1]})_{1}$ is therefore the same as
that of $w_{1}$, and the row pattern of $(w^{[n-1]})_{2}$ is the same
as that of $w_{2}$. Any descent in $(w^{[n-1]})_{1}$ or
$(w^{[n-1]})_{2}$ would give rise to a descent in $w^{[n-1]}$, so the
matrices $(w^{[n-1]})_{1}$ and $(w^{[n-1]})_{2}$ must be
ascending. Thus $(w^{[n-1]})_{1}=(w_{1})^{[n-1]}$ and
$(w^{[n-1]})_{2}=(w_{2})^{[n-1]}$,
establishing~\eqref{eq:(ascending)}.\par In a similar way, we let
$w_{[n-1]}=(w_{[n-1]})_{1}*(w_{[n-1]})_{2}$, and we claim that
\begin{equation}
w_{[n-1]}=(w_{1})_{[n-1]}*(w_{2})_{[n-1]},\label{eq:(permutation)}
\end{equation}
 that is $(w_{[n-1]})_{1}=(w_{1})_{[n-1]}$ and
$(w_{[n-1]})_{2}=(w_{2})_{[n-1]}$.
Indeed, $\sigma_{\mathrm{0}}$ is a homomorphism, and so
\begin{multline*}
(w_{1},w_{2}) =\sigma_{0}(w)
  =\sigma_{0}(w^{[n-1]}w_{[n-1]})=\sigma_{0}(w^{[n-1]})\sigma_{0}(w_{[n-1]})\\
  =((w^{[n-1]})_{1}(w_{[n-1]})_{1},(w^{[n-1]})_{2}(w_{[n-1]})_{2}),
\end{multline*}
and thus $w_{1}=(w^{[n-1]})_{1}(w_{[n-1]})_{1}$ and
$w_{2}=(w^{[n-1]})_{2}(w_{[n-1]})_{2}$.  Using \eqref{eq:(ascending)}
we obtain
\[
w_{1} = (w_{1})^{[n-1]} (w_{1})_{[n-1]} = (w^{[n-1]})_{1}
(w_{1})_{[n-1]}=(w^{[n-1]})_{1}(w_{[n-1]})_{1},
\]
whence $(w_{[n-1]})_{1}=(w_{1})_{[n-1]}$. The equality
$(w_{[n-1]})_{2}=(w_{2})_{[n-1]}$ is proved in the same way. This
proves \eqref{eq:(permutation)}.\par

Recall that $l(w)=l(w^{[n-1]})+l(w_{[n-1]})$;
see~\eqref{equ:add.length}.  Since $\coxinv+\coxnsp=l-\coxneg$
(see~\eqref{equ:length}) and
$\coxneg(w)=\coxneg(w^{[n-1]})+\coxneg(w_{[n-1]}) =
\coxneg(w^{[n-1]})$, this implies that
\[
(\coxinv+\coxnsp)(w)=(\coxinv+\coxnsp)(w^{[n-1]})+(\coxinv+\coxnsp)(w_{[n-1]}).
\]
Lemma~\ref{lem:L}~\eqref{enu:L(2)} and the
equalities \eqref{eq:(ascending)} and \eqref{eq:(permutation)} now
imply that
%\begin{align*}
%{L(w^{[n-1]})+L(w_{[n-1]})}&=\coxneg((w_{1})^{[n-1]})+(\coxinv+\coxnsp)(w^{[n-1%]})\\
% & \quad-(\coxinv+\coxnsp)((w_{1})^{[n-1]})-(\coxinv+\coxnsp)((w_{2})^{[n-1]})\%\
% & \quad+\coxneg((w_{1})_{[n-1]})+(\coxinv+\coxnsp)(w_{[n-1]})\\
% & \quad-(\coxinv+\coxnsp)((w_{1})_{[n-1]})-(\coxinv+\coxnsp)((w_{2})_{[n-1]})\%\
% &
%=\coxneg(w_{1})+(\coxinv+\coxnsp)(w)-(\coxinv+\coxnsp)(w_{1}
%)-(\coxinv+\coxnsp)(w_{2})\\&=L(w).
%\end{align*}
\begin{multline*}
{L(w^{[n-1]})+L(w_{[n-1]})}=\\\coxneg((w_{1})^{[n-1]})+(\coxinv+\coxnsp)(w^{[n-1]})
-(\coxinv+\coxnsp)((w_{1})^{[n-1]})-(\coxinv+\coxnsp)((w_{2})^{[n-1]})\\ +\coxneg((w_{1})_{[n-1]})+(\coxinv+\coxnsp)(w_{[n-1]})-(\coxinv+\coxnsp)((w_{1})_{[n-1]})-(\coxinv+\coxnsp)((w_{2})_{[n-1]})\\ =\coxneg(w_{1})+(\coxinv+\coxnsp)(w)-(\coxinv+\coxnsp)(w_{1}
)-(\coxinv+\coxnsp)(w_{2})=L(w).
\end{multline*}
\end{proof}

\subsection{A second additivity result for $L$}\label{subsec:second}
We now consider how the statistic $L$ behaves with respect to
parabolic factorisations of the form $w =w^{[i-1]_0}w_{[i-1]_0}$,
where $i\in[n-1]$. Even if $w\in\firste$, it is not necessarily true
that $L(w) = L(w^{[i-1]_0}) + L(w_{[i-1]_0})$. A counter-example is
given by $i=2$ and $w=[-5,2,1,-4,3]\in\mathcal{E}_{5}$.  Here
$D(w)=\{0,2,3\}$ and $L(w)=7$.  But $L(w^{\{0,1\}}) = L([2,5,1,-4,3])
= 6$ and $L(w_{\{0,1\}}) = L([-2,1,3,4,5]) = 2$.

The following result establishes additivity of $L$ under this kind of
parabolic factorisation under additional conditions.

\begin{prop}\label{prop:even second}
 Suppose that $n$ is even and $w\in\firste$ has even descent
 type~$D(w)$. Let $\imin\in [n-1]$ be an even integer such that
 $\imin\leq \min\{(D(w)\cup \{n\})\setminus\{0\}\}$, that is $w(1) <
 \dots < w(\imin)$.  Then
\[
L(w)=L(w^{[\imin-1]_{0}})+L(w_{[\imin-1]_{0}}).
\]
 \end{prop}

\begin{proof}
Write the factorisation $w=w^{[\imin-1]_{0}}w_{[\imin-1]_{0}}$
as
\[
w=\left(\begin{array}{cc|ccc}
 & \\
A & \, & \, & M\\
 & \\
\end{array}\right)=\left(\begin{array}{cc|ccc}
 & \\
B & \, & \, & M\\
 & \\
\end{array}\right)\left(\begin{array}{ccc|ccc}
 &  & \\
 & \overline{A} &  &  & 0\\
 &  & \\
\hline  &  & \\
 & 0 &  &  & \Id_{n-\imin}\\
 &  & \\
\end{array}\right)=w^{[\imin-1]_{0}}w_{[\imin-1]_{0}},
\]
where $A\in\Mat(n\times \imin,\Z)$ comprises the first $\imin$ columns
of~$w$, $M\in\Mat(n\times(n-\imin),\Z)$ comprises the last $n-\imin$
columns and $\Id_{n-\imin}$ denotes the identity matrix of
size~$n-\imin$. We now describe the matrices $B$ and $\overline{A}$.
Define
\[
f:[\imin]\longrightarrow[\imin],\quad
\kappa \mapsto \#\{(r,s)\in[n]\times[\imin]\mid w_{rs}\neq0\wedge r\leq
i(\kappa)\}.
\]
Informally speaking, $f$ enumerates the rows in $A$ containing a
non-zero entry, so that for $\kappa\in[\imin]$, the non-zero entry of $w$
in column $\kappa$ lies in the $f(\kappa)$-th non-zero row in
$A$. Since each column of $A$ contains exactly one non-zero entry, the
function $f$ is a bijection. Given this definition, $B$ is the
$n\times i$-matrix whose $(i_w(j),f(j))$-entry is $1$ for
$j\in[\imin]$, and all other entries zero, and $\overline{A}$ is the
$i\times i$ ascending matrix whose $(f(j),j)$-entry
is~$w_{i_w(j),j}$.\par

Recall the formula $L(w)=a(w)+b(w)+2c(w)$ given in
Lemma~\ref{lem:L}~\eqref{enu:L(1)}. Using the assumptions that $n$ and
$D(w)$ are even, we will show that the functions $a$,$b$ and $c$ are
each additive over the
factorisation~$w=w^{[\imin-1]_{0}}w_{[\imin-1]_{0}}$. Clearly
$a(w)=a(A)+a(M)$,
$a(w_{[\imin-1]_{0}})=a(\overline{A})=a(A)$ and
$a(w^{[\imin-1]_{0}})=a(M)$, so $a$ is additive. It is easy to verify
that $b(w_{[\imin-1]_{0}})=b(\overline{A})=b(A)$,
$c(w_{[\imin-1]_{0}})=c(\overline{A})=c(A)$ and that $b(B)=c(B)=0$.
Hence the respective additivity of $b$ and $c$ is equivalent to
\begin{align}
b(w)-b(A)-b(M) & =b(w^{[\imin-1]_{0}})-b(B)-b(M)\label{eq:add-b-c-1}\\
\intertext{and}
c(w)-c(A)-c(M) & =c(w^{[\imin-1]_{0}})-c(B)-c(M).\label{eq:add-b-c-2}
\end{align}
 These two equations can be interpreted in the following way. Let
 $V\in\Mat(n,\Z)$ be a matrix with at most one non-zero entry in each
 column, such as $w$, $w^{[\imin-1]_{0}}$ or $w_{[\imin-1]_{0}}$.  Suppose that
$V$ has the
 form
\[
V = \left( V_1 \mid V_2 \right),
\]
 where $V_{1}\in\Mat(n\times \imin,\Z)$ consists of the first $\imin$
 columns of $V$, and $V_{2}\in\Mat(n\times(n-\imin),\Z)$ consists of
 the remaining $n-\imin$ columns. Then, by Definition~\ref{def:abc},
\begin{align*}
b(V)-b(V_{1})-b(V_{2}) &=
\#\{(j_{1},j_{2})\in[\imin]\times[n-\imin]\mid
i_{V}(j_{1})>i_{V}(j_{2}),\;j_{1}\not\equiv j_{2} \bmod (2)\}
\end{align*} and
\begin{multline*} c(V) -
c(V_1) - c(V_2) \\= \#\{(j_1,j_2) \in[\imin]\times[n-\imin] \mid
V_{i_V(j_2),j_2}=-1,\;i_V(j_1) < i_V(j_2),\; j_1 \not\equiv j_2 \bmod
(2)\}
\end{multline*}
Informally, the value $b(V)-b(V_{1})-b(V_{2})$ is equal to the
contribution to $b$ given by column pairs $(j_{1},j_{2})$ such that
$j_{1}$ denotes a column of $V_{1}$ and $j_{2}$ denotes a column of
$V_{2}$. Similar considerations hold for the function $c$.\par To
prove the equations \eqref{eq:add-b-c-1} and \eqref{eq:add-b-c-2} it
therefore suffices to establish a bijection
$\phi:[\imin]\rightarrow[\imin]$, inducing a bijection between the
columns of $A$ and the columns of $B$ such that
\begin{gather}
j \equiv \phi(j) \bmod{(2)},\label{eq:condition1}\\ i_w(j)>i_w(k)
\Longleftrightarrow i_{B}(\phi(j))>i_w(k) \text{ for all $j\in[\imin]$,
  $\imin<k\leq n$.}\label{eq:condition2}
\end{gather}

We consider $w$ as obtained from the ascending matrix $w^{[n-1]}$ by
column permutations, given by $w_{[n-1]}$. Since both $n$ and $D(w)$
are even, Lemma~\ref{lem:evenperm-block} implies that $w_{[n-1]} =
(w_{[n-1]})_{1}*(w_{[n-1]})_{1}$. This implies that $w$ is obtained
from $w^{[n-1]}$ by permuting pairs of adjacent columns
of~$w^{[n-1]}$, indexed by pairs of the form $(2j-1,2j)$, for
$j\in[n/2]$. Note that $w^{[n-1]}\in\mathcal{C}_{n,0}$. We may
therefore apply Lemma~\ref{lem:ascending-struct} to the column pairs
of $w^{[n-1]}$, and any statement about these column pairs remains
true for the column pairs of the submatrix $A$ of $w$.
Assume that $w^{[n-1]}$ is non-trivial; otherwise, there is nothing to
prove. To define the bijection $\phi$, we consider a pair $(2j-1,2j)$
for $j\in[\imin/2]$. We distinguish two cases:

\medskip
\emph{Case $w(2j)>0$.} Here Lemma~\ref{lem:ascending-struct} implies
that $f(2j) \equiv 0\bmod (2)$ and $f(2j-1)\equiv 1 \bmod (2)$, and in
this case we set
\begin{equation*}
\phi(2j-1) = f(2j-1), \quad
\phi(2j) = f(2j).
\end{equation*}

\emph{Case $w(2j)<0$.} Here Lemma~\ref{lem:ascending-struct} implies that
$w(2j-1) = w(2j)-1$ and thus $f(2j-1) =
  f(2j)+1$. Therefore, if $f(2j-1) \equiv 1 \bmod (2)$ then $f(2j) \equiv 0
\bmod (2)$, and in this case we set
%\begin{align*}
%  \phi(2j-1) &= f(2j-1) \\
%  \phi(2j) &= f(2j).
%\end{align*}
\begin{equation*}
  \phi(2j-1) = f(2j-1), \quad
  \phi(2j) = f(2j).
\end{equation*}
  On the other hand, if $f(2j-1) \not\equiv 1 \bmod (2)$ then $f(2j) \not\equiv
0 \bmod (2)$. In other words, in this case we have $f(2j-1) \equiv 0 \bmod (2)$
and $f(2j) \equiv 1 \bmod (2)$, and
  we set
\begin{equation*}
  \phi(2j-1) = f(2j), \quad
  \phi(2j) = f(2j-1).
\end{equation*}
%\begin{align*}
%  \phi(2j-1) &= f(2j)\\
%  \phi(2j) &= f(2j-1).
%\end{align*}
Note that this last case is the only one where $\phi$ does not agree
with $f$.

By definition the bijection $\phi$ satisfies condition
 \eqref{eq:condition1}. Moreover, in the cases where $\phi(j)=f(j)$
 we have $i_{B}(\phi(j))=i_{B}(f(j))=i_w(j)$, since, as noted previously,
the non-zero entry in column $f(j)$ in the matrix $B$ lies in row $i_w(j)$.
Thus, condition \eqref{eq:condition2} is satisfied whenever $\phi(j)=f(j)$.
Finally, in the case where $(2j-1,2j)$ is a column pair such that
$\phi(2j-1)=f(2j)$ and $\phi(2j)=f(2j-1)$, we have $f(2j-1)=f(2j)+1$, so
$$i_{B}(\phi(2j))+1=i_{B}
(f(2j-1))+1=i_w(2j-1)+1=i_w(2j)=i_B(f(2j))=i_B(\phi(2j-1)).$$
Thus, for $k$ such that $\imin<k\leq n$, we have
\begin{align*}
i_w(2j-1)>i_w(k) &\Longleftrightarrow i_w(2j)>i_w(k)\\
\Longleftrightarrow i_B(\phi(2j-1))>i_w(k)&\Longleftrightarrow i_B(\phi(2j))>i_w(k).
\end{align*}
Therefore condition \eqref{eq:condition2} is satisfied also in this case.

We have thus established the existence of a bijection $\phi$ with
the required properties, and this finishes the proof.
\end{proof}

\subsection{Proof of Case \eqref{3} of Theorem~\ref{thm:Main}}\label{subsec:case3}
\foreignlanguage{english}{For $a,b\in\N_{0}$ such that $a\geq b$, the
  \emph{$X$-binomial coefficient} is defined as
\[
\binom{a}{b}_{X}=\frac{(\underline{a})!}{(\ul{a-b})!(\underline{b})!}\in\Z[X].
\]
More generally, for $n\in\N$ and
$I=\{i_{1},\dots,i_{l}\}_{<}\subseteq[n-1]_{0}$, the
\emph{$X$-multinomial coefficient} is
\[
\binom{n}{I}_{X}=\binom{n}{i_{l}}_{X}\binom{i_{l}}{i_{l-1}}_{X}\cdots\binom{i_{2}}{i_{1}}_{X}\in\Z[X].
\]}
It is well known that
\begin{equation}\label{equ:stanley}
 \sum_{w\in S_n^{I^{\mathrm{c}}}}X^{l(w)} = \binom{n}{I}_X ;
\end{equation}
see, for instance, \cite[Proposition~1.3.17]{Stanley/97}.

\begin{lem}\label{lem:Evenperm}Suppose that $n$ and $I\subseteq[n-1]_{0}$ are
  even. Then
\[
\sum_{w\in S_{n}^{I^{\mathrm{c}}}}(-1)^{l(w)}X^{L(w)}=\binom{n/2}{I/2}_{X^{2}}.
\]
 \end{lem}

 \begin{proof} By arguing as in the proof of Lemma~\ref{lem:cb-supp},
   we may argue that the sum is supported on the set
   $S_n^{I^{\mathrm{c}}}\cap \mathcal{C}_{n,0}$. Let $w\in
   S_n^{I^{\mathrm{c}}}\cap\cbe$ and write $w=w_{1}*w_{2}$.  By
   Lemma~\ref{lem:evenperm-block} we have $w_{1}=w_{2}\in
   S_{n/2}^{(I/2)^{\mathrm{c}}}$
   and~$l_{S_n}(w)=4l_{S_{n/2}}(w_{1})$. By Lemma~\ref{lem:L}
   \eqref{enu:L(2)} we have $L(w)=l(w)-2l(w_{1})$. (Here and in the
   sequel we suppress subscripts in the notation for various Coxeter
   length functions.) Using~\eqref{equ:stanley} we obtain
  \begin{align*}
    \sum_{w\in S_{n}^{I^{\mathrm{c}}}}(-1)^{l(w)}X^{L(w)} &=
    \sum_{w\in
      S_{n}^{I^{\mathrm{c}}}\cap\cbe}(-1)^{l(w)}X^{l(w)-2l(w_{1})}\\
    &=\sum_{w_{1}\in
      S_{n/2}^{(I/2)^{\mathrm{c}}}}(-1)^{4l(w_{1})}X^{4l(w_{1})-2l(w_{1})}
    \\&
=\sum_{w_{1}\in
      S_{n/2}^{(I/2)^{\mathrm{c}}}}X^{2l(w_{1})}=\binom{n/2}{I/2}_{X^{2}}.
\end{align*}
 \end{proof}

\begin{prop} \label{pro:Even}
 Conjecture~\ref{conjecture} holds when both $n$ and
 $I\subseteq[n-1]_{0}$ are even, that is, in this case
\[
\sum_{w\in
  B_{n}^{I^{\mathrm{c}}}}(-1)^{l(w)}X^{L(w)}=\binom{n/2}{I/2}_{X^{2}}\frac{(\underline{1})(\underline{3})\cdots(\underline{n-1})}{(\underline{1})(\underline{3})\cdots(\underline{i_{1}-1})}.
\]
\end{prop}
\begin{proof} By Corollary~\ref{cor:Monochrome-supp}, we have
\[
\sum_{w\in
  B_{n}^{(I_{0})^{\mathrm{c}}}}(-1)^{l(w)}X^{L(w)}=\sum_{w\in\firste^{(I_{0})^{
      \mathrm{c}}}}(-1)^{l(w)}X^{L(w)}.
\]
Recall that $i_1 = \min(I\cup\{n\})$. By the two additivity results
for $L$ established in Propositions~\ref{prop:first} and
\ref{prop:even second}, this may be written as
\begin{align*}
\sum_{w\in B_{n}^{(I_{0})^{\mathrm{c}}}}(-1)^{l(w)}X^{L(w)} & =\left(\sum_{w\in
B_{n}^{[n-1]}}(-1)^{l(w)}X^{L(w)}\right)\left(\sum_{w\in
S_{n}^{I^{\mathrm{c}}}}(-1)^{l(w)}X^{L(w)}\right)\\
 & =\left(\sum_{w\in B_{n}^{I^{\mathrm{c}}}}(-1)^{l(w)}X^{L(w)}\right)\left(\sum_{w\in B_{i_{1}}^{[i_{1}-1]}}(-1)^{l(w)}X^{L(w)}\right).
\end{align*}
The proposition now follows from Proposition~\ref{pro:Ascending} (twice) and
Lemma~\ref{lem:Evenperm}.
\end{proof}
 \foreignlanguage{english}{\begin{acknowledgement*} This research was
     supported by EPSRC grant EP/F044194/1. \end{acknowledgement*} }

\def\cprime{$'$}
\providecommand{\bysame}{\leavevmode\hbox to3em{\hrulefill}\thinspace}
\providecommand{\MR}{\relax\ifhmode\unskip\space\fi MR }
\providecommand{\MRhref}[2]{%
  \href{http://www.ams.org/mathscinet-getitem?mr=#1}{#2}
}
\providecommand{\href}[2]{#2}

\end{document}